\documentclass[12pt,a4paper,twoside]{amsart}

\usepackage{amsmath}
\usepackage{amssymb}
\usepackage{amscd,amsthm}
\usepackage{amsmath}
\usepackage{xcolor}
\usepackage{amsthm}
\usepackage{amssymb}
\usepackage{amsfonts}
\usepackage{dsfont}
\usepackage[arrow, matrix, curve]{xy}
\usepackage{mathtools}      
\usepackage{comment}

\definecolor{c2}{RGB}{50,135,90}
\definecolor{c1}{RGB}{10,100,155}
\usepackage[pagebackref=true, colorlinks=true, citecolor=c1, linkcolor=c2]{hyperref}

\theoremstyle{plain}
\newtheorem{prop}{Proposition}[section]
\newtheorem{theo}[prop]{Theorem}

\newtheorem{coro}[prop]{Corollary}
\newtheorem{lem}[prop]{Lemma}
\newtheorem{introtheo}{Theorem}[section]

\theoremstyle{definition}
\newtheorem{dfn}[prop]{Definition}

\newtheorem{ex}[prop]{Example}

\DeclareMathOperator{\vol}{\mathrm{vol}}

\DeclareMathOperator{\conv}{\mathrm{conv}}
\DeclareMathOperator{\Hom}{{\mathrm Hom}}
                  
\newcommand*{\defeq}{\coloneqq}
\def\C{{\mathbb C}}
\def\P{{\mathbb P}}
\def\Z{{\mathbb Z}}
\def\N{{\mathbb N}}
\def\R{{\mathbb R}}
\def\Q{{\mathbb Q}}
\def\T{{\mathbb T}}

\def\calo{\mathcal{O}}

\usepackage{tikz}
\usepackage{tikz-3dplot}
\usetikzlibrary{calc}
\usetikzlibrary{backgrounds}
\usepackage{subfigure}

\definecolor{cof}{RGB}{219,144,71}
\definecolor{pur}{RGB}{186,146,162}
\definecolor{greeo}{RGB}{91,173,69}
\definecolor{greet}{RGB}{52,111,72}

\usepackage{pgfplots}
\pgfplotsset{compat=newest}
\usepgfplotslibrary{patchplots}
\begin{document}
\title[Stringy $E$-functions of toric Fano $3$-folds and their applications]
{Stringy $E$-functions of canonical toric Fano threefolds and their applications}

\author[Victor Batyrev]{Victor Batyrev}
\address{Mathematisches Institut, Universit\"at T\"ubingen, 
Auf der Morgenstelle 10, 72076 T\"ubingen, Germany}
\email{victor.batyrev@uni-tuebingen.de}

\author[Karin Schaller]{Karin Schaller}
\address{Mathematisches Institut, Universit\"at T\"ubingen, 
Auf der Morgenstelle 10, 72076 T\"ubingen, Germany}
\email{karin.schaller@uni-tuebingen.de}

%\subjclass{}
%\keywords{}

\begin{abstract}
Let $\Delta$ be a $3$-dimensional lattice polytope
containing exactly one interior lattice point.
We give a simple combinatorial formula for computing the stringy $E$-function
of the $3$-dimensional canonical toric Fano  variety $X_{\Delta}$ 
associated with the polytope $\Delta$.
Using the stringy Libgober-Wood identity and our formula, we 
generalize the well-known combinatorial identity 
$\sum_{\theta \preceq \Delta \atop \dim (\theta) =1}  v(\theta) \cdot v(\theta^*) = 24$ 
holding in the case of $3$-dimensional reflexive polytopes $\Delta$. 
\end{abstract}

%\date{\today}

\maketitle

\thispagestyle{empty}
\section*{Introduction}
\setcounter{section}{0}
Let $N$ be a free abelian group of rank $d$.  A $d$-dimensional convex
polytope $\Delta \subseteq N_{\R} \defeq N \otimes \R$ having vertices in $N$ 
(i.e., $\Delta = \conv(\Delta \cap N)$) is called
{\em lattice polytope}.  A lattice polytope $\Delta$ is called 
{\em canonical Fano} if  the origin $0 \in N$ is the only lattice point in its interior 
$\Delta^{\circ}$, i.e., $\Delta^{\circ} \cap N=\{0\}$. Let $X_{\Delta}$  be the toric variety given by the {\em spanning fan} of
a canonical Fano polytope $\Delta$, i.e., $X_{\Delta}$ is defined by the fan 
$\Sigma_{\Delta} \defeq \{ \sigma_{\theta} \, \vert \, \theta \preceq \Delta \}$ whose cones $\sigma_{\theta} \defeq \R_{\geq 0} \theta$ are spanned by all faces $\theta$ of $\Delta$.
It is known that the toric variety $X_{\Delta}$ has at worst canonical singularities \cite{Rei83} and the anticanonical divisor $-K_{X_{\Delta}}$ of $X_{\Delta}$ is an ample $\Q$-Cartier divisor \cite{Dan78}, 
i.e., $X_{\Delta}$ is a {\em canonical toric Fano variety}. 
 One can show that any $d$-dimensional canonical toric Fano variety can be obtained from a $d$-dimensional canonical Fano polytope $\Delta$ as above and that there is a bijection between
$d$-dimensional canonical toric Fano varieties up to isomorphism and $d$-dimensional canonical Fano  polytopes $\Delta \subseteq N_{\R}$ considered up to an isomorphism  of
the lattice $N$ \cite{Kas10}. The classification of $d$-dimensional canonical toric
Fano varieties is known only for $d \leq 3$.

\smallskip
Let $M \defeq \Hom( N, \Z)$ be the dual lattice and $M_{\R}$ the corresponding
real vector space together with the natural pairing $\langle \cdot, \cdot \rangle : M_{\R} \times N_{\R} \to \R$.
For any $d$-dimensional convex polytope $\Delta \subseteq N_{\R}$ containing $0 \in N$ in its
interior, we consider the {\em dual polytope} 
\[ \Delta^* \defeq \left\{ y \in M_{\R}\,  \middle\vert\,  \left<y,x \right> \geq -1  \; \forall x \in \Delta
\right\} \subseteq M_{\R}. \]
There exists a natural bijection (duality) between $k$-dimensional faces $\theta \preceq \Delta$
of $\Delta$
and $(d-k-1)$-dimensional {\em dual faces} $\theta^* \preceq \Delta^*$ of the dual polytope
$\Delta^*$  defined as
\[\theta^* \defeq \{  y \in \Delta^*  \,  \vert\,  \left<y,x \right> =  -1  \; \forall x \in \theta \}.    \] 
A $d$-dimensional lattice polytope $ \Delta \subseteq N_{\R}$  containing $0 \in N$ in its
interior is called {\em reflexive}
if all vertices of the dual polytope $\Delta^*$ belong to the dual lattice $M$,
i.e., if $\Delta^*$ is also a lattice polytope \cite{Bat94}.   
If $\Delta$ is reflexive, then 
$\Delta^*$ is also reflexive and $(\Delta^*)^* = \Delta$. 
Any facet $\theta \preceq \Delta$ of a reflexive polytope  $\Delta$ is defined 
 by an equation $\left< m, x \right > = -1$ for some lattice vertex $m \in M$ of the dual reflexive polytope $\Delta^*$. This means that each facet $\theta$ of a reflexive polytope $\Delta$ has lattice distance $1$ to the origin $0 \in N$ and the origin is the only interior lattice point of $\Delta$.  In particular, every reflexive polytope is a canonical Fano polytope. The converse is not true.
If $\Delta$ is a canonical Fano polytope, then the {\em lattice distance} $n_{\theta}$
from a $(d-1)$-dimensional face $\theta \preceq \Delta$
to the origin can be larger than $1$, i.e., the facet $\theta$ is defined by an 
equation $- \left<m,x \right > = n_\theta > 1$ for some primitive lattice point  
$m \in M$. 
One can show that
any canonical Fano polytope  of dimension $d \leq 2$ is always reflexive, but the latter
is not true if $d \geq 3$.
There exist exactly $4,\!319$
isomorphism classes of $3$-dimensional reflexive polytopes 
classified by Kreuzer and Skarke \cite{KS98a}. However, there exist 
much more (exactly $674,\!688$)  isomorphism classes of $3$-dimensional canonical 
Fano polytopes classified by Kasprzyk \cite{Kas10}.

\smallskip
If $\Delta \subseteq N_{\R}$ is an arbitrary $d$-dimensional lattice polytope,
then we define $v(\Delta)$ to be the {\em normalized volume},  
i.e., the positive integer $d!\cdot \vol_d (\Delta)$,
where $\vol_d(\Delta)$ denotes the $d$-dimensional volume of $\Delta$ with respect to the lattice $N$. 
Similarly, we define
the positive integer $v(\theta) \defeq k! \cdot \vol_k(\theta)$  for a $k$-dimensional face
$\theta \preceq \Delta$, where $\vol_k(\theta)$ denotes 
the $k$-dimensional volume of $\theta$ with respect to the sublattice $\text{span}(\theta) \cap N$. 
If $\Delta \subseteq N_{\R}$ is a  $d$-dimensional
polytope having vertices in $N_{\Q} \defeq N \otimes \Q$, i.e., $\Delta$ is a rational
polytope, then we can similarly define the positive rational number 
$v(\theta)$ for any  $k$-dimensional face
$\theta \preceq \Delta$. For this purpose, we consider an integer $l$ such that $l\Delta$ is 
a lattice polytope and define for a $k$-dimensional face $\theta \preceq \Delta$ its volume as
$v(\theta) \defeq \frac{1}{l^k} v(l \theta)$.

\bigskip
For any $3$-dimensional reflexive polytope $\Delta$ one has the well-known combinatorial identity
\begin{align} \label{3dimrefl}
24=  \sum_{\theta \preceq \Delta \atop \dim(\theta)=1}
v(\theta) \cdot v\left(\theta^* \right) \tag{$\star$}
\end{align}
\cite[Theorem 4.3]{Haa05} (Corollary \ref{cor24}). A possible proof of this identity can be obtained from  the fact that the Euler
number of a smooth $K3$-surface equals $24$. In addition, 
one considers a generic affine hypersurface $Z_{\Delta} \subseteq (\C^*)^3$ 
defined by a Laurent polynomial $g_{\Delta} \in \C[ x_1^{\pm 1}, x_2^{\pm 1}, x_3^{\pm 1}]$ 
with the reflexive Newton polytope $\Delta$ in the $3$-dimensional torus 
$(\C^*)^3 \subseteq X_{\Delta^*}$ and applies the following two statements:

\begin{itemize}
\item The projective closure  $\overline{Z_{\Delta}} \subseteq X_{\Delta^*}$
of $Z_{\Delta}$ is a $K3$-surface with at worst canonical
singularities \cite[Theorem 4.1.9]{Bat94};
\item The number
\[ \sum_{\theta \preceq \Delta \atop \dim(\theta)=1}
v(\theta) \cdot v\left(\theta^* \right) \]
equals the stringy Euler number 
of  $\overline{Z_{\Delta}}$ \cite[Corollary 7.10]{BD96}.
\end{itemize}
One of our results is a generalization of the combinatorial identity (\ref{3dimrefl}) for arbitrary 
$3$-dimensional canonical Fano polytopes $\Delta$. First, we consider a $3$-dimensional lattice polytope $\Delta \subseteq N_{\R}$   such that
a generic affine hypersurface $Z_{\Delta} \subseteq (\C^*)^3$  defined
by a Laurent polynomial $f_{\Delta} \in \C[ x_1^{\pm 1}, x_2^{\pm 1}, x_3^{\pm 1}] $ 
with the Newton polytope $\Delta$  is birational
to a smooth $K3$-surface. Such a lattice polytope $\Delta$ must be 
a $3$-dimensional canonical Fano polytope, because
the number of interior lattice points in $\Delta$ equals the geometric genus of the affine
surface
$Z_{\Delta} \subseteq (\C^*)^3$ \cite{Kho78} and it equals  $1$ for a smooth 
$K3$-surface. 
Moreover, one can show that the above condition on $\Delta$ 
is satisfied if and only if $\Delta$ is contained
in some reflexive polytope $\Delta'$. Those  $3$-dimensional canonical
Fano polytopes are called {\em almost reflexive}. 
There exist many equivalent characterizations
 of $3$-dimensional almost reflexive polytopes $\Delta$ \cite{Bat18}, e.g., $\Delta \subseteq N_{\R}$ is almost 
reflexive if and only if the convex hull $[\Delta^*] \defeq \conv(\Delta^* \cap M)$ 
of all lattice points in the dual polytope $\Delta^* \subseteq M_{\R}$ is a reflexive polytope.
 Using the list of all $3$-dimensional
canonical Fano polytopes \cite{Kas10}, Kasprzyk showed that among all $674,\!688$ $3$-dimensional
canonical Fano polytopes there exist exactly $9,\!089$
lattice polytopes $\Delta$  which are not almost reflexive. 
In other words, there exist exactly $665,\!599$ almost reflexive polytopes of dimension $3$. 
Any almost reflexive polytope $\Delta$ allows us to
construct a family of smooth $K3$-surfaces using the 
toric variety $X^\Delta$ associated with the normal fan $\Sigma^{\Delta}$ of $\Delta$. 
We recall that the {\em normal fan} 
$\Sigma^{\Delta}\defeq \{\sigma^{\theta} \, \vert \, \theta \preceq \Delta \}$ consists of 
cones $\sigma^\theta  := \R_{\geq 0} \theta^*$ corresponding to faces $\theta \preceq \Delta$,
where $\sigma^{\theta}$ is the cone generated by the dual face $\theta^* \preceq \Delta^*$ 
of the dual polytope $\Delta^*$ and 
$\dim(\sigma^{\theta})= \dim( \theta^*) +1 = d - \dim(\theta)$.

\smallskip
A generalization of the combinatorial formula (\ref{3dimrefl}) for 
an arbitrary $3$-dimensional almost reflexive polytope $\Delta$ is the following statement:

\begin{introtheo}\label{theo1}
Let $\Delta \subseteq N_{\R}$ be an arbitrary $3$-dimensional almost reflexive polytope.  Then
\[24 =
v(\Delta) - \sum_{\theta \preceq \Delta \atop \dim (\theta) =  2}
\frac{1}{n_{\theta}}\cdot v(\theta) + 
\sum_{\theta \preceq \Delta \atop   \dim (\theta) =1}
v(\theta) \cdot v(\theta^*) . \]
\end{introtheo}

The proof of this new combinatorial formula  is built around
the more general combinatorial formula
\[e_{\rm str}(Y) = \sum_{\theta \preceq \Delta \atop \dim (\theta) \geq 1}
(-1)^{\dim (\theta) -1} \, v(\theta) \cdot v(\sigma^{\theta} \cap \Delta^*) \]
\cite[Theorem 4.11]{Bat18} (Theorem \ref{str-euler1})
computing the
stringy Euler number $e_{\rm str}(Y)$ of a $(d-1)$-dimensional Calabi-Yau variety $Y$ 
with at worst canonical singularities
which is birational to a generic $(d-1)$-dimensional affine hypersuface $Z_{\Delta} \subseteq (\C^*)^d$ 
defined by a generic Laurent polynomial $f_{\Delta} \in \C[x_1^{\pm 1}, \cdots, x_d^{\pm 1}] $  
with the Newton polytope $\Delta$. If $ d=3$, as we mentioned,  such a $(d-1)$-dimensional Calabi-Yau
variety $Y$ exists if and only if $\Delta$ is a $3$-dimensional almost reflexive polytope.

\smallskip
It is remarkable that the same identity holds for all $9,\!089$
lattice polytopes $\Delta$  which are not almost reflexive, i.e., it holds for all $674,\!688$ 
$3$-dimensional canonical Fano polytopes $\Delta$.

\begin{introtheo}\label{theo2}
Let $\Delta \subseteq N_{\R}$ be an arbitrary $3$-dimensional canonical Fano polytope.  Then
\[24 =
v(\Delta) - \sum_{\theta \preceq \Delta \atop \dim (\theta) =  2}
\frac{1}{n_{\theta}}\cdot v(\theta) + 
\sum_{\theta \preceq \Delta \atop   \dim (\theta) =1}
v(\theta) \cdot v(\theta^*) . \]
\end{introtheo}

The proof of this more general statement uses  a combinatorial formula for computing the
stringy $E$-function $E_{\rm str}(X_\Delta;u,v)$ of  a $3$-dimensional canonical 
toric Fano variety $X_\Delta$.

\begin{introtheo}\label{theo3}
Let $X_{\Delta}$ be a $3$-dimensional canonical toric Fano variety
defined by a $3$-dimensional canonical Fano polytope $\Delta \subseteq N_{\R}$. Then
the stringy $E$-function of $X_{\Delta}$  can be computed as
\begin{align*}
E_{\rm str}\left(X_\Delta;u,v\right) = \left((uv)^3+1\right)+ r\cdot \left((uv)^2+ (uv)\right)  + \!\!\!
\sum_{\theta \preceq \Delta \atop \dim(\theta)=2,  n_{\theta}>1}  \!\!\!\!\!
v(\theta) \cdot  \left(\sum_{k=1}^{n_{\theta}-1}(uv)^{\frac{k}{n_{\theta}}+1} \right)  ,
\end{align*}
where $r \defeq \left\vert \Delta \cap N\right\vert -4$.
\end{introtheo}

In order to prove Theorem \ref{theo2}, we combine the above formula for the
stringy $E$-function $ E_{\rm str}\left(X_\Delta;u,v\right)$
with our combinatorial version of the stringy Libgober-Wood identity in terms of
generalized stringy Hodge numbers and intersection products of stringy Chern classes
\cite[Theorem 4.3]{BS17}.

\subsection*{Organisation of the paper.}
In Section \ref{section1}, we launch remaining and necessary definitions as well as 
explain important relationships focused on almost reflexive
polytopes. In addition, we prove Theorem \ref{theo24}.

In Section \ref{section2}, we investigate the stringy $E$-function of
arbitrary $3$-dimensional toric Fano
varieties with at worst canonical singularities and prove Theorem \ref{strefct}. 

In Section \ref{section3}, we apply Theorem \ref{strefct} in order 
to prove Theorem \ref{theo2.2}
using in addition a combinatorial version of the stringy Libgober-Wood identity.

\subsection*{Acknowledgments.}
We would like to thank Tom Coates, Alessio Corti, Alexander Kasprzyk, and Hannah Markwig for inspiring discussions. 

%%%%%%%%%%%%%%%%%%%%%%%%%%%%%%%%%%
\section{Stringy Euler numbers and Calabi-Yau hypersurfaces} \label{section1}

The orbifold or stringy Euler number has been introduced by
Dixon, Harvey, Vafa, and Witten \cite{DHVW85} as a natural topological
invariant  motivated by string theory for singular Calabi-Yau varieties.
Further investigations of mirror symmetry for singular Calabi-Yau varieties
led to the notion of stringy $E$-functions  \cite{Bat98}. We recall their definition below.

\smallskip
Let $X$ be  a $d$-dimensional normal projective
$\Q$-Gorenstein variety
with at worst {\em log-terminal 
singularities}, i.e.,
the canonical divisor $K_X$ of $X$ is a $\Q$-Cartier
divisor and for some
desingularization $\rho : X' \rightarrow X$
of $X$, whose exceptional locus
is a union of  smooth irreducible divisors $D_1, \ldots, D_s$
with only normal crossings,  one has
$K_{X'}=\rho^*  K_X  + \sum_{i=1}^s a_iD_i$
for some rational numbers $a_i >-1$  for all $1 \leq i \leq s$.
Moreover, we
define  $D_J$ to be the intersection $\cap_{j \in J} D_j$ for any nonempty subset $J \subseteq I \defeq\{1,\ldots,s\}$
and set $D_{\emptyset}\defeq X'$. We remark that the subvariety $D_J \subseteq X'$
is either empty or a smooth projective subvariety of $X'$ of codimension $|J|$.
The {\em stringy $E$-function} of $X$ is a rational algebraic function in two
variables $u, v$ defined as
\begin{align*}
E_{\rm str}(X; u, v) \defeq \sum_{\emptyset \subseteq J \subseteq I}
E(D_J; u,v) \prod_{j \in J} \left( \frac{uv-1}{(uv)^{a_j +1} -1} -1 \right)
\end{align*}
 \cite[Definition 2.1]{Bat98},
where $E(D_J; u,v) = \sum_{0 \leq p,q \leq d-|J|} (-1)^{p+q} h^{p,q}(D_J) u^p v^q$  is
the generating polynomial for the Hodge numbers $h^{p,q}(D_J)$ of
the smooth projective variety $D_J$. The {\em stringy Euler number}  of $X$
is defined as
\begin{align*}
e_{\rm str}(X) \defeq \lim_{u,v \to 1}  E_{\rm str}(X; u, v) =
\sum_{\emptyset \subseteq J \subseteq I} c_{d-\lvert J \rvert}(D_J)
\prod_{j \in J} \left( \frac{-a_j}{a_j+1}\right).
\end{align*}
The restriction on the
singularities  of $X$ is used to proof the independence of  $E_{\rm str}(X; u,$ $v)$
on the choice of the desingularization $\rho$ \cite[Theorem 3.4]{Bat98}.
As a special case, this formula implies the equality 
$E_{\rm str}(X;u,v) = E(X;u,v)$
if $X$ is smooth and the equality 
$E_{\rm str}(X;u,v) = E(X';u,v)$
if $\rho: X' \to X$ is a crepant desingularization of $X$, i.e., if  all 
rational numbers $a_i$ are zero.

%\smallskip
It is very convenient to consider stringy invariants while investigating singular varieties $X$ that appear in birational geometry and the minimal model program. It is known that a minimal model $X'$ of a projective algebraic variety $X$ with at worst
log-terminal singularities is not uniquely determined, because sometimes 
there exists another 
minimal model $X''$ of $X$, which is obtained from $X'$ by a sequence of flops.
Nevertheless, the stringy $E$-functions of  $X'$ and $X''$ are the same, i.e.,
$E_{\rm str}(X'; u,v) = E_{\rm str}(X''; u,v)$.  The same equality $E_{\rm str}(X'; u,v) = E_{\rm str}(X''; u,v)$ holds if the minimal
model $X'$  admits a crepant birational morphism $ X' \to X''$ to a projective
variety $X''$  with at worst canonical $\Q$-Gorenstein singularities. In particular, the stringy
$E$-function of a minimal birational model of $X$ and the stringy $E$-function of  a
canonical model of $X$ are the same, i.e., they are independent of the choice of
these birational models.

\smallskip
In this section, we are interested in the stringy Euler numbers $e_{\rm str}(Y)$ of minimal and canonical models of algebraic varieties $Y$ that are birational to a generic
affine hypersurface $Z_{\Delta} \subseteq (\C^*)^d$
defined  by a generic Laurent polynomial $f_{\Delta} \in \C[x_1^{\pm 1}, \ldots, x_d^{\pm 1}]$ whose
Newton polytope is a given $d$-dimensional lattice polytope $\Delta \subseteq N_{\R}$.
In this case, the stringy Euler number $e_{\rm str}(Y)$ is uniquely determined by the lattice polytope $\Delta$ and it
can be computed by a combinatorial formula.  The generic
open condition on the coefficients of the
Laurent polynomial $f_{\Delta}$  can be formulated in a more precise form:

\begin{dfn}
Let $\Delta \subseteq N_{\R}$ be a $d$-dimensional lattice  polytope
and $Z_{\Delta} \subseteq \T_d \cong (\C^*)^d$ an affine hypersurface, which
is the zero set of a Laurent polynomial 
\[ f_{\Delta}(x) = \sum_{ n \in   \Delta \cap \Z^d} a_n x^n  \in
\C[x_1^{\pm 1}, \ldots, x_d^{\pm 1}] \]
with the Newton polytope $\Delta$ and some sufficiently general coefficients $a_n \in \C$.
The Laurent polynomial $f_{\Delta}$ and the affine hypersurface $Z_{\Delta}$
are called
{\em $\Delta$-nondegenerate} if for every
face $\theta \preceq \Delta$ the zero locus $Z_{\theta} \defeq
\{ x \in \T_d \,\vert \, f_\theta(x) =0\}$ of the  $\theta$-part of $f_{\Delta}$ is
 empty or a smooth affine hypersurface in the
$d$-dimensional algebraic
torus $\T_d$.
\end{dfn}

\begin{dfn} 
Let $\Delta \subseteq N_{\R}$ be a $d$-dimensional lattice polytope
containing the origin $0 \in N$  in its interior. The lattice
polytope $\Delta$
is called
{\em almost pseudoreflexive}
if the convex hull $[\Delta^*] = \conv(\Delta^* \cap M)$ of all lattice points in
the dual polytope $\Delta^*$ also contains the origin $0 \in M$
in its interior. An almost pseudoreflexive polytope $\Delta$ is called {\em pseudoreflexive}
if  $\Delta = [[\Delta^*]^*]$. 
If $\Delta$ is an almost pseudoreflexive polytope, then $[\Delta^*]$ is always a 
pseudoreflexive polytope. We call an almost pseudoreflexive
polytope $\Delta$  {\em almost reflexive} if $[\Delta^*]$ is a reflexive polytope.
\end{dfn}

Using a result of Skarke \cite{ska96}, one obtains
that every pseudoreflexive polytope $\Delta$ of dimension $d \leq 4$ is
reflexive. 
In particular, every almost pseudoreflexive polytope of dimension $d \leq 4$ is
almost reflexive. There exists an equivalent description of almost reflexive
polytopes $\Delta$ of dimension $d \leq 4$: $\Delta$ is almost pseudoreflexive if and only if
$\Delta$ contains $0 \in N$ in its interior and is contained in some reflexive polytope $\Delta'$
\cite[Proposition 3.4 and Theorem 3.12]{Bat18}.

\smallskip
The above classes of lattice polytopes $\Delta$  are related to  some (singular) Calabi-Yau varieties:

\begin{dfn}
Let $Y$ be a $(d-1)$-dimensional normal irreducible projective algebraic variety
with trivial canonical divisor $K_Y$. We call $Y$ a {\em Calabi-Yau variety}
if $Y$  has at worst
canonical Gorenstein singularities and $H^i(Y,\calo_Y)= 0$ ($0 < i < d-1$).
\end{dfn}

Almost pseudoreflexive polytopes $\Delta$ of dimension $d$ are 
in the following sense closely related 
to Calabi-Yau varieties: 

\begin{prop}[{\cite[Theorem 2.23]{Bat18}}] \label{prop1}
Let $\Delta \subseteq N_{\R} $ be a $d$-dimensional lattice polytope
containing the origin $0 \in N$  in its interior. A $\Delta$-nondegenerate 
affine hypersurface $Z_{\Delta}  \subseteq \T^d$ is birational
to a $(d-1)$-dimensional Calabi-Yau variety $Y$ if and only the polytope
$\Delta$ is almost pseudoreflexive. In particular, in the case  $d =3$  
a $\Delta$-nondegenerated affine
surface $Z_{\Delta}  \subseteq \T^3$ is birational to a $K3$-surface
if and only if $\Delta $ is an almost reflexive polytope.
\end{prop}

In this paper, we consider the following three running examples of $3$-dimensional canonical Fano polytopes $\Delta_1$, $\Delta_2$, and $\Delta_3$:

\begin{ex} 
Set $\Delta_1 \defeq \conv(e_1,e_2,e_3,-e_1-e_2-e_3)$,
$\Delta_2 \defeq \conv(e_1,e_2,e_3,-e_1 - e_2 - 2 e_3)$, and 
$\Delta_3 \defeq \conv(e_1,e_2,e_3,- 5e_1 - 6e_2 - 8e_3)$
(Figure \ref{fig:picturestrefct}).  All these polytopes
 contain the origin $0 \in \Z^3$ as an unique interior lattice point. The
canonical Fano polytope $\Delta_1$ is  reflexive, but the canonical Fano polytopes $\Delta_2$ and
$\Delta_3$ are not reflexive. One can check that
the lattice polytope $[\Delta_i^*] = \conv(\Delta_i^* \cap N)$ contains the origin in
 its interior (Figure \ref{fig:P(1,1,1,1)2}, \ref{fig:P(1,1,1,2)2}, 
 and \ref{fig:pictureP(1,5,6,8)2}) if and only if $i \in \{ 1,2 \}$.
 So the lattice simplices $\Delta_1$, $\Delta_2$ are almost reflexive, but the lattice
 simplex $\Delta_3$ is not almost reflexive.
\end{ex}

%%%%%%%%%%%%%%%%%%%%%%%%%%%%%%%%%%%%%%%%%%%%%%%%%%%%%
\begin{figure}[h!]
\centering
% 1
\subfigure[$\Delta_1$ reflexive.]{
\label{fig:P(1,1,1,1)}
\scalebox{0.35}{
\includegraphics{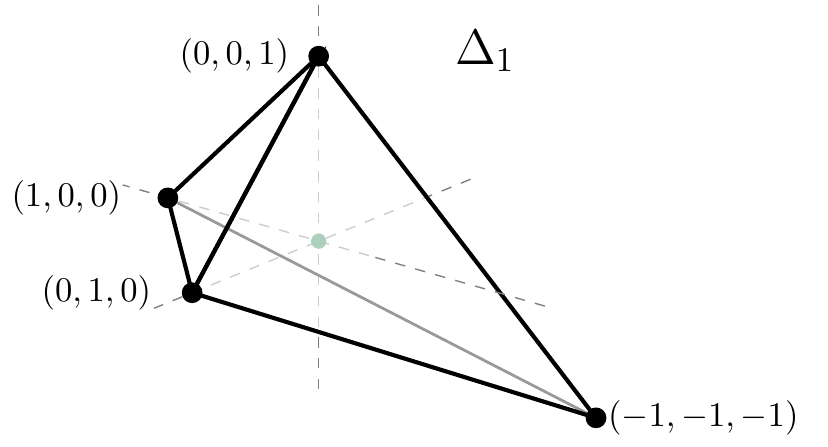}
}
}
% 2
\subfigure[$\Delta_2$ almost reflexive.]{
\label{fig:P(1,1,1,2)}
\scalebox{0.4}{
\includegraphics{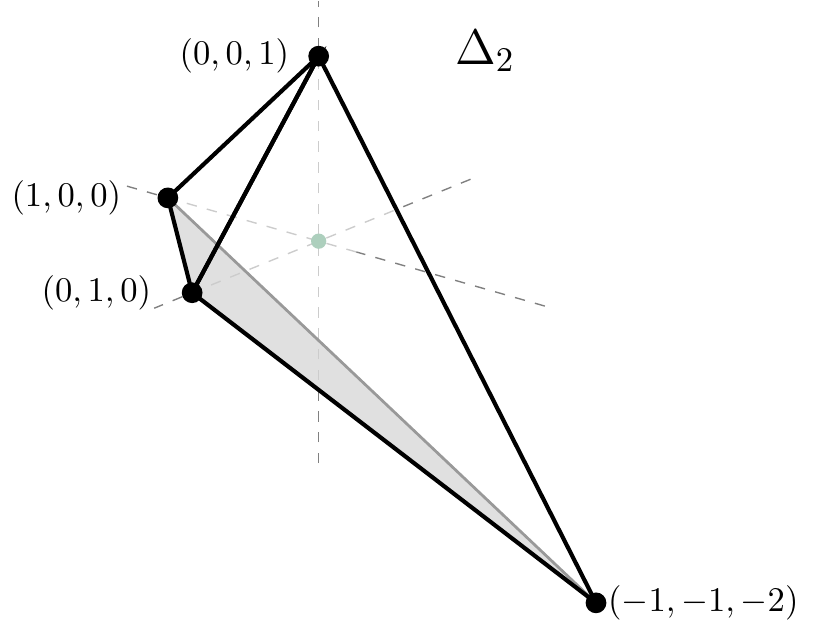}
}
}
% 3
\subfigure[$\Delta_3$ not almost reflexive.]{
\label{fig:P(1,5,6,8)}
\scalebox{0.5}{
\includegraphics{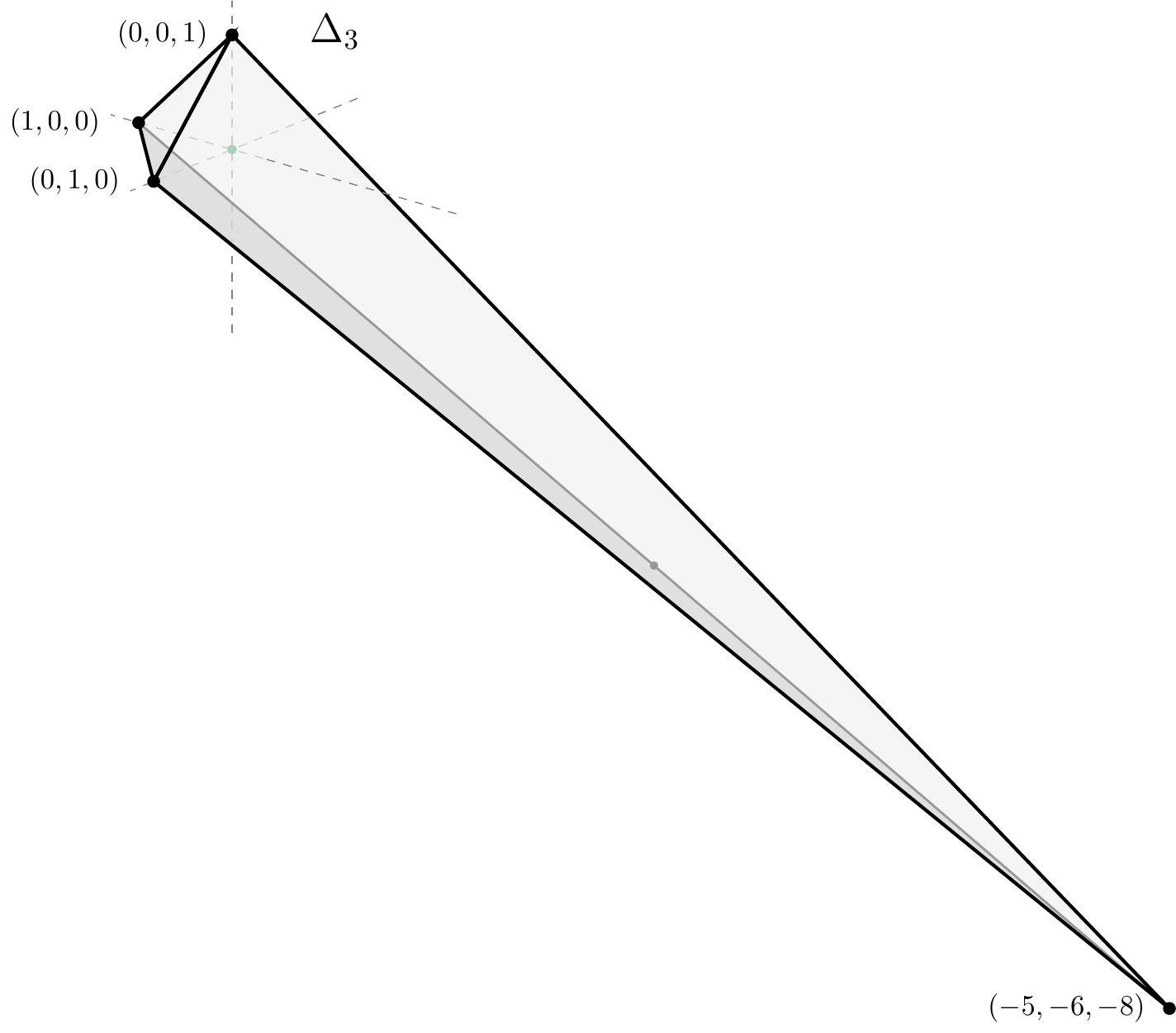}
}
} \hspace{0pt}
\caption[]{Three types of canonical Fano polytopes in dimension $3$.}
\label{fig:picturestrefct}
\end{figure}
%%%%%%%%%%%%%%%%%%%%%%%%%%%%%%%%%%%%%%%%%%%%%%%%%%%%%

Now we formulate a main result of \cite{Bat18}:

\begin{theo}[{\cite[Theorem 4.11]{Bat18}}] \label{str-euler} 
Let $\Delta \subseteq N_\R$  be an arbitrary $d$-dimensional
almost pseudoreflexive polytope.
Denote by $Y$ a canonical
Calabi-Yau model of a $\Delta$-nondegenerate affine
hypersurface $Z_{\Delta}  \subseteq \T_d$.
Then the stringy Euler number of $Y$ can be computed by the 
combinatorial formula
\[    e_{\rm str}(Y) = \sum_{\theta \preceq \Delta \atop \dim (\theta) \geq 1}
(-1)^{\dim (\theta) -1} v(\theta) \cdot v(\sigma^{\theta} \cap \Delta^*), \]
where $\sigma^\theta = \R_{\geq 0} \theta^*$ is the cone in the normal fan $\Sigma^{\Delta}$ of $\Delta$  corresponding to the dual face $\theta^* \preceq \Delta^*$ 
of the dual polytope $\Delta^* \subseteq M_\R$.
\end{theo}

The combinatorial formula for the stringy Euler number in Theorem \ref{str-euler} can 
be rewritten in the following equivalent form:

\begin{theo} \label{str-euler1}
Let $\Delta \subseteq N_\R$  be an arbitrary $d$-dimensional
almost pseudoreflexive polytope.
Denote by $Y$ a canonical
Calabi-Yau model of a $\Delta$-nondegenerated affine
hypersurface $Z_{\Delta}  \subseteq \T_d$.
Then the stringy Euler number of $Y$ can be computed by the
combinatorial formula
\[e_{\rm str}(Y) = v(\Delta) - \sum_{\theta \preceq \Delta \atop \dim (\theta) =  d-1}
 \frac{1}{n_{\theta}} \cdot v(\theta)  + 
\sum_{\theta \preceq \Delta \atop 1 \leq  \dim (\theta) \leq d-2}
(-1)^{\dim (\theta) -1}   v(\theta) \cdot v(\theta^*) . \]
\end{theo}

\begin{proof}
Let $\theta \preceq \Delta$ be a face of $\Delta$. 
If  $\theta = \Delta$,  then $\sigma^{\theta} \cap \Delta^* = \{0\}$ and 
$v(\sigma^{\theta} \cap \Delta^*)=1$. If $\theta$ is a $(d-1)$-dimensional face 
of $\Delta$ and $u_{\theta}$ denotes the primitive inward-pointing facet normal of $\theta$, 
then $\sigma^{\theta} \cap \Delta^* = \conv(0, 1/n_{\theta} \cdot u_{\theta})$ 
and $v(\sigma^{\theta} \cap \Delta^*) = 1/n_{\theta} $.  If $\theta \preceq \Delta$ is 
face of $\Delta$ of dimension $k$ $(1 \leq k \leq d-2)$, then $ \sigma^{\theta} \cap \Delta^* $ 
is a $(d-k)$-dimensional pyramid with vertex $0 \in M$ over the $(d-k-1)$-dimensional 
dual face $\theta^*$. By definition of the dual polytope $\Delta^*$, the lattice distance 
between $0$ and $\theta^*$ equals $1$. So we get 
$v(\sigma^{\theta} \cap \Delta^*) = v(\theta^*)$.  
Now it remains to apply the formula of Theorem \ref{str-euler}. 
\end{proof} 

We apply this theorem in the case $d=3$ to get a new 
combinatorial identity for the Euler number $24$:

\begin{introtheo} \label{theo24}
Let $\Delta \subseteq N_\R$  be an arbitrary $3$-dimensional almost reflexive polytope.
Then 
\[24 =
v(\Delta) - \sum_{\theta \preceq \Delta \atop \dim (\theta) =  2}
\frac{1}{n_{\theta}}\cdot v(\theta) + 
\sum_{\theta \preceq \Delta \atop   \dim (\theta) =1}
v(\theta) \cdot v(\theta^*) . \]
\end{introtheo}

\begin{proof}
A canonical Calabi-Yau model $Y$ of a $\Delta$-nondegenerate affine
hypersurface $Z_{\Delta}  \subseteq \T_3$ is a $K3$-surface 
with at worst canonical singularities
(Proposition \ref{prop1}).
Its  minimal crepant desingularization $\tilde{Y}$  is a smooth $K3$-surface 
and by Noether's Theorem
$$2 = \chi(\tilde{Y},\calo_{\tilde{Y}})= \frac1{12} (c_1(\tilde{Y})^2 + c_2(\tilde{Y}))$$
the stringy Euler number  $e_{\rm str}(Y) $  equals $c_2(\tilde{Y}) = 24$.
Using the formula for  $e_{\rm str}(Y) $  from Theorem \ref{str-euler1},
we get the required identity.
\end{proof}

The statement of Theorem \ref{theo24} specializes to  the well-known 
formula for reflexive polytopes in dimension $3$
\cite[Theorem 4.3]{Haa05} (Equation \eqref{3dimrefl}):

\begin{coro} \label{cor24}
 Let $\Delta \subseteq N_{\R}$ be a $3$-dimensional reflexive polytope. Then
 \[24 = \sum_{\theta \preceq \Delta \atop \dim(\theta)=1}
v(\theta) \cdot v\left(\theta^*  \right) .\]
\end{coro}

\begin{proof}
We apply Theorem \ref{theo24} and get
\begin{align*}
24 =
v(\Delta) - \sum_{\theta \preceq \Delta \atop \dim (\theta) =  2}
\frac{1}{n_{\theta}} \cdot v(\theta)  +
\sum_{\theta \preceq \Delta \atop   \dim (\theta) =1}
 v(\theta) \cdot v(\theta^*).
\end{align*}
Since $\Delta$ is reflexive, we 
have $n_\theta =1$ for all facets $\theta \preceq \Delta$ and  
\[  v(\Delta) -\sum_{\theta \preceq \Delta \atop \dim(\theta) = 2}
 \frac{1}{n_{\theta}}\cdot v(\theta)  =  v(\Delta) -\sum_{\theta \preceq \Delta \atop \dim(\theta) = 2}
v(\theta) =0.   \]
The equation from above simplifies to
\[ 24 = \sum_{\theta \preceq \Delta \atop \dim(\theta)=1} v(\theta) \cdot v(\theta^*).\]
\end{proof}

\begin{ex}
For the reflexive polytope $\Delta_1 = \conv(e_1,e_2,e_3,-e_1-e_2-e_3)$
(Figure \ref{fig:P(1,1,1,1)21} and \ref{fig:P(1,1,1,1)}) we get  the identity
 \begin{align*}
 \sum_{\theta \preceq \Delta_1 \atop \dim(\theta)=1}
v(\theta) \cdot v\left(\theta^* \right)
= 6 \cdot \left( 1! \cdot 1) \cdot(1! \cdot 4 \right)
=24,
\end{align*}
because the dual polytope equals 
\[\Delta_1^*=\conv((3,-1,-1),(-1,3,-1),(-1,-1,3),(-1,-1,-1))\]
and $\dim(\theta^*)=1$ with $v\left(\theta^* \right) = 1! \cdot 4=4$ for every $1$-dimensional face
$\theta \preceq \Delta_1$.
\end{ex}

%%%%%%%%%%%%%%%%%%%%%%%%%%%%%%%%%%%%%%%%%%%%%%%%%%%%%
\begin{figure} [h!]
\centering
% 1
\subfigure[]{
\label{fig:P(1,1,1,1)21}
\scalebox{0.6}{
\includegraphics{pictureP1111}
}
}
% 2
\subfigure[]{
\label{fig:P(1,1,1,1)dual22}
\scalebox{0.5}{
\includegraphics{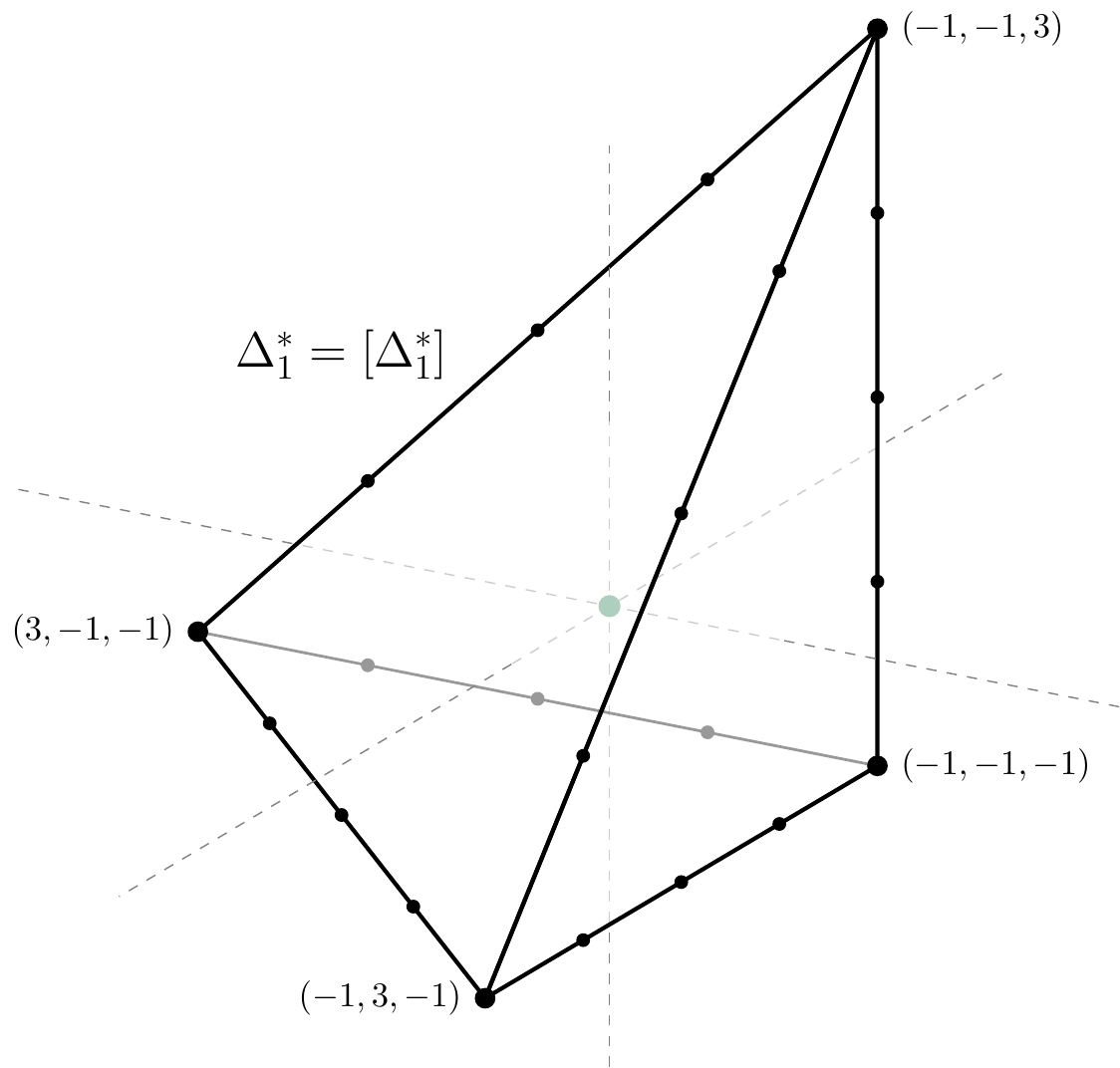}
}
}
\caption[]{Illustration of a reflexive polytope.}
\label{fig:P(1,1,1,1)2}
\end{figure}
%%%%%%%%%%%%%%%%%%%%%%%%%%%%%%%%%%%%%%%%%%%%%%%%%%%%%

\begin{ex}
Consider the almost reflexive polytope
 $\Delta_2 = \conv(e_1,e_2,e_3,-e_1-e_2-2e_3)$
(Figure \ref{fig:P(1,1,1,2)21} and \ref{fig:P(1,1,1,2)}).
Its dual polytope $\Delta_2^*$ is rational:
$$\Delta_2^* = \conv((4,-1,-1),(-1,4,-1),(-1,-1,3/2), (-1,-1,-1)).$$
There exists exactly one $2$-dimensional face $($highlighted in gray in Figure \ref{fig:P(1,1,1,2)21}$)$
of $\Delta_2$ having lattice distance $2 >1$ to the origin. The combinatorial identity
from Theorem \ref{theo24} holds, because
 \begin{align*}
& v(\Delta_2) - \sum_{\theta \preceq \Delta_2 \atop \dim (\theta) =  2}
 \frac{1}{n_{\theta}}\cdot v(\theta) + 
\sum_{\theta \preceq \Delta_2 \atop   \dim (\theta) =1}
  v(\theta) \cdot v(\theta^*) \\
&=  5
- (\frac12 \cdot 1+ 1 \cdot 1+ 1 \cdot 1+1 \cdot 1 )
+ ( 1 \cdot \frac52+ 1 \cdot  \frac52 + 1 \cdot  \frac52 +1 \cdot  5 + 1 \cdot  5 +1 \cdot  5 ) \\
&= 5 - 3.5 + 22.5 =24.
\end{align*}
\end{ex}

%%%%%%%%%%%%%%%%%%%%%%%%%%%%%%%%%%%%%%%%%%%%%%%%%%%%%
\begin{figure} [h!]
\centering
% 1
\subfigure[]{
\label{fig:P(1,1,1,2)21}
\scalebox{0.6}{
\includegraphics{pictureP1112}
}
}
% 2
\subfigure[]{
\label{fig:P(1,1,1,2)dual22}
\scalebox{0.5}{
\includegraphics{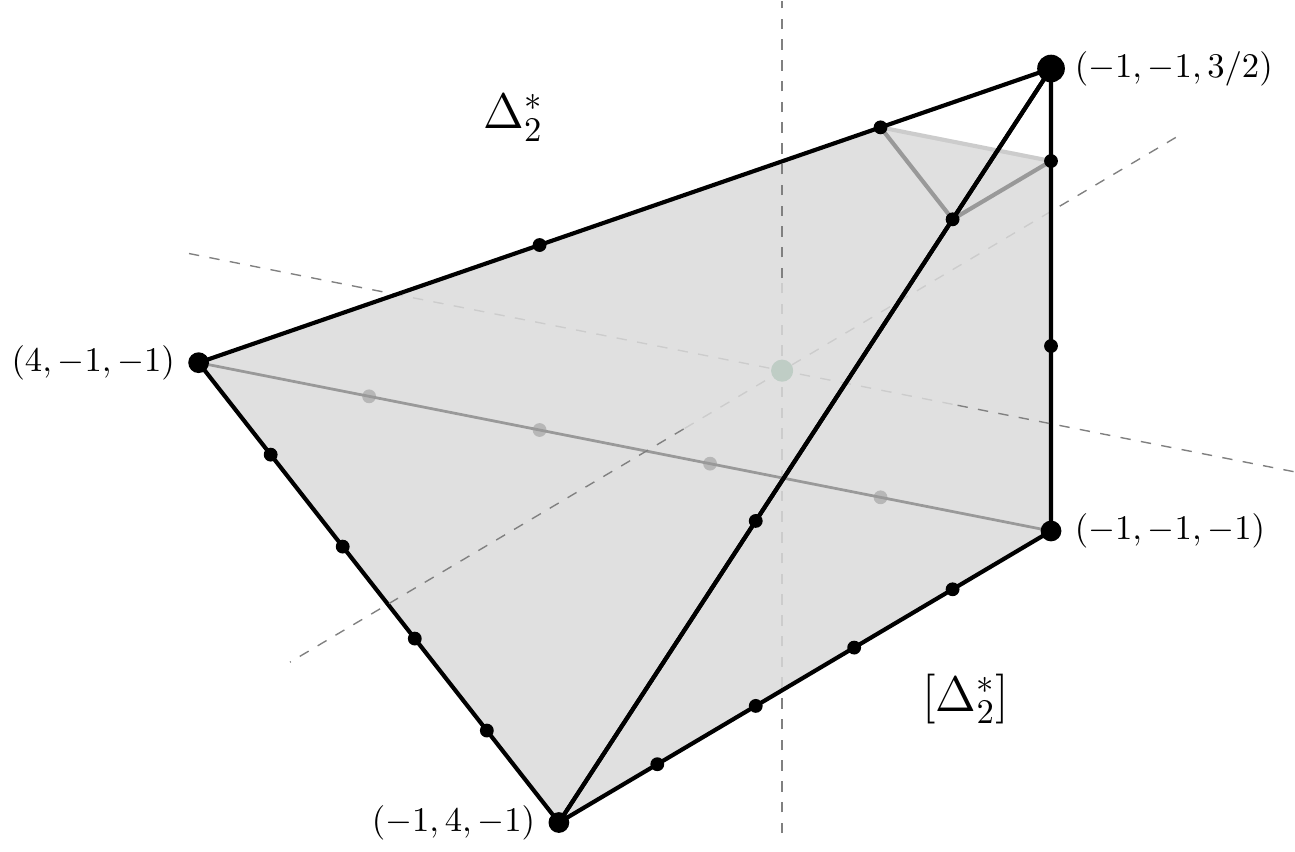}
}
}
\caption[]{Illustration of an almost reflexive polytope.}
\label{fig:P(1,1,1,2)2}
\end{figure}
%%%%%%%%%%%%%%%%%%%%%%%%%%%%%%%%%%%%%%%%%%%%%%%%%%%%%

%%%%%%%%%%%%%%%%%%%%%%%%%%%%%%%%%%
\section{Stringy E-functions of canonical toric Fano threefolds} \label{section2}

Let $X$ be a $d$-dimensional projective $\Q$-Gorenstein toric variety
associated with a fan $\Sigma$ of rational polyhedral cones in $N_{\R}$. 
 We denote by $\Sigma(k)$  the set of all $k$-dimensional cones in $\Sigma$. 
 We choose $\Sigma'$ to be a simplicial subdivision of the
fan $\Sigma$ such that $\Sigma'(1) \subseteq \Sigma(1)$ and define 
$\kappa_{\Sigma}$ to be the $\Sigma$-piecewise linear function corresponding to the anticanonical
divisor $-K_X$ of $X$
(i.e., $\kappa_{\Sigma}$ has value $-1$ on every primitive ray generator of a $1$-dimensional
cone $\rho \in \Sigma(1)$).   Moreover, we define
$\square_{\sigma}^{\circ}\defeq \sigma^{\circ} \cap \square_{\sigma}$
for any cone  $\sigma \in \Sigma'$,
 where $\sigma^{\circ}$ denotes the
relative interior of $\sigma$
and $\square_{\sigma}$ the parallelepiped
spanned by the primitive ray generators of the cone $\sigma$.

\smallskip
There is a general formula for the stringy $E$-function $E_{\rm str}\left(X;u,v\right)$
of the $d$-dimensional projective $\Q$-Gorenstein
toric variety $X$:  

\begin{prop}[{\cite[Proposition 4.1]{BS17}}] \label{prop2}
Let $X$ be a $d$-dimensional projective $\Q$-Gorenstein
toric variety
associated with a fan $\Sigma$ . Then
the stringy $E$-function of $X$
can be computed as the
finite sum
\[E_{\rm str}\left(X;u,v\right) = \sum_{\sigma \in \Sigma'} (uv-1)^{d-\dim(\sigma)}
\sum_{n \in \square_{\sigma}^{\circ} \cap N} (uv)^{\dim(\sigma)+\kappa_{\Sigma} (n)}.\]
In particular, 
\[E_{\rm str}\left(X;u,v\right) = \sum_{\alpha} \psi_{\alpha}(\Sigma) (uv)^{\alpha} \]
for some nonnegative integers $ \psi_{\alpha}(\Sigma)$ and some nonnegative 
rational numbers $\alpha$.  
\end{prop}

A much more precise computation of the stringy $E$-function $E_{\rm str}\left(X;u,v\right)$ has been obtained if $X$ is a {\em toric log del Pezzo
surface}, i.e., a normal projective toric surface
with at worst log-terminal singularities
and an ample anticanonical $\Q$-Cartier divisor $-K_X$ \cite{KKN10}. 
Such a toric surface is defined by a {\em $LDP$-polygon} $\Delta \subseteq N_{\R}$,
i.e., a $2$-dimensional lattice polytope containing
the origin $0 \in N$ in its interior such that all vertices of $\Delta$
are primitive lattice points in $N$. In this case,  $\Sigma= \Sigma_{\Delta}$ is the spanning fan in the $2$-dimensional space $N_{\R}$ consisting of cones over faces of the $LDP$-polygon $\Delta$. 
Using Proposition \ref{prop2},  we got

\begin{theo}[{\cite[Corollary 4.5]{BS17}}] \label{ldpestr}
Let $X$ be a toric log del Pezzo surface defined by a $LDP$-polygon $\Delta \subseteq N_{\R}$. Then 
the stringy $E$-function of $X$ can be computed as 
\[E_{\rm str}\left(X;u,v\right)
= ((uv)^2+1) +r \cdot  (uv) + \sum_{n \in \Delta^{\circ}\setminus \{0\}}
\left( (uv)^{2+\kappa_{\Delta}(n)} + (uv)^{-\kappa_{\Delta}(n)} \right), \]
where  $r\defeq   \vert \partial \Delta \cap N \vert -2$
and $\kappa_{\Delta}$  is the $\Sigma_{\Delta}$-piecewise linear function corresponding to the anticanonical divisor $-K_X$ of $X$ with value $-1$ on the whole boundary $\partial \Delta$ of $\Delta$.
\end{theo}

\begin{ex}
Set $\Delta \defeq \conv(e_1,e_2,-e_1-e_2)$ 
and $\Delta_m \defeq \conv(e_1,e_2,-e_1-me_2)$
for an integer $m \in \N$
(Figure \ref{fig:picture12d} and \ref{fig:picture22d}), 
i.e., $\Delta$ is a reflexive polygon and 
$\Delta_m$ a $LDP$-polygon. Therefore, $X_{\Delta}= \P_2$ 
 and
$X_{\Delta_m}$ are
toric log del Pezzo surfaces with
\begin{align*} 
E_{\rm str}\left(X_{\Delta};u,v\right) 
&= (uv)^2 + (uv) + 1 
\end{align*}
and 
\begin{align*} 
E_{\rm str}\left(X_{\Delta_3};u,v\right) 
&= ((uv)^2+1) +(uv) + 
\big( (uv)^{\frac{4}{3}} + (uv)^{\frac{2}{3}} \big)
\end{align*}
by using Theorem \ref{ldpestr},
because $\sum_{0 < n \leq \frac{3}{2}} 
\big( (uv)^{2-\frac{2}{3}n} + (uv)^{\frac{2}{3}n} \big)=  (uv)^{\frac{4}{3}} + (uv)^{\frac{2}{3}} $.
\end{ex}

%%%%%%%%%%%%%%%%%%%%%%%%%%%%%%%%%%%%%%%%%%%%%%%%%%%%%
\begin{figure}[h!]
\centering 
% 1 
\subfigure[Reflexive polygon.]{
\label{fig:picture12d}
\scalebox{0.7}{ 	
\includegraphics{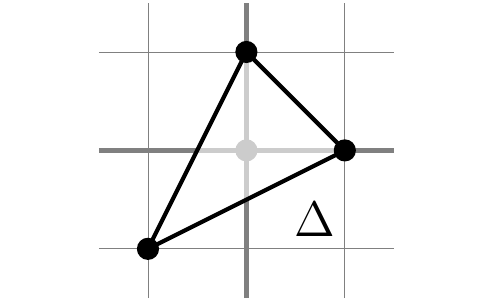}
} 
}
% 2 
\subfigure[$LDP$-polygon (here: m=3).]{
\label{fig:picture22d}
\scalebox{0.7}{ 
\includegraphics{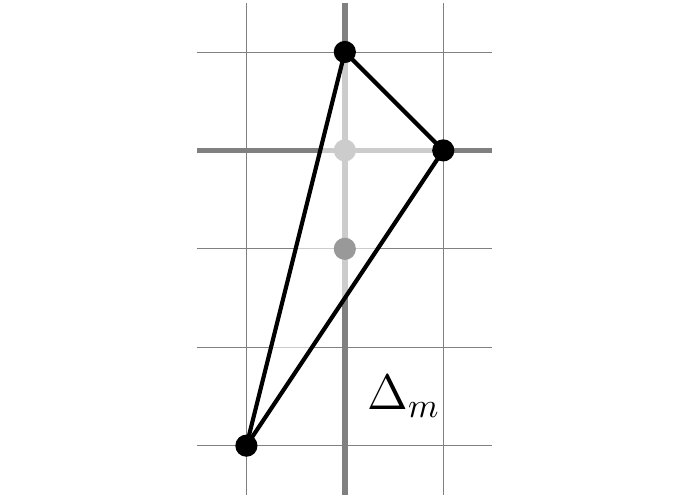}
} 
} 
\caption[]{Reflexive polygons and $LDP$-polygons.}
\label{fig:picture122d}
\end{figure}
%%%%%%%%%%%%%%%%%%%%%%%%%%%%%%%%%%%%%%%%%%%%%%%%%%%%%

\medskip
Let $\Delta \subseteq N_{\R}$ be a $3$-dimensional canonical Fano polytope, $\Sigma_{\Delta}$
the associated spanning fan of $\Delta$, and $X \defeq X_{\Delta}$ the corresponding
toric Fano threefold
with at worst canonical singularities.
We are interested in a similar combinatorial formula to compute the stringy $E$-function
$E_{\rm str}\left(X;u,v\right)$ of such a toric Fano
threefold $X$
by making solely use of information about the  lattice polytope $\Delta$
that yields the appearing coefficients $\psi_{\alpha}(\Sigma_{\Delta})$ and rational powers of
$(uv)^\alpha$ in $E_{\rm str}\left(X;u,v\right)$
in a direct way.
To be precise, the following formula depends 
primarily on the volumes $v(\theta)$ and lattice distances $n_{\theta}$ to the origin of $2$-dimensional faces $\theta \preceq \Delta$:

\setcounter{introtheo}{2}
\begin{introtheo}\label{strefct}
Let $\Delta \subseteq N_{\R}$ be a $3$-dimensional canonical Fano polytope. 
Then the stringy $E$-function of the corresponding canonical toric Fano threefold $X \defeq X_\Delta$
can be computed as
\begin{align*}
E_{\rm str}\left(X;u,v\right) = \left((uv)^3+1\right)+ r\cdot \left((uv)^2+ (uv)\right)  + \!\!\!
\sum_{\theta \preceq \Delta \atop \dim(\theta)=2,  n_{\theta}>1}  \!\!\!\!\!
v(\theta) \cdot \left( \sum_{k=1}^{n_{\theta}-1}(uv)^{\frac{k}{n_{\theta}}+1}  \right)
\end{align*}
with $r \defeq \left\vert \Delta \cap N\right\vert -4$. 
\end{introtheo}

\begin{proof}
Applying Proposition \ref{prop2}, 
we know that the stringy $E$-function of $X$ is given as
\begin{align*} 
E_{\rm str}\left(X;u,v\right) &= \sum_{\sigma \in \Sigma_{\Delta}'} (uv-1)^{3-\dim(\sigma)} 
\sum_{n \in \square_{\sigma}^{\circ} \cap N} (uv)^{\dim(\sigma)+\kappa_{\Delta} (n)} ,
\end{align*}
where $\Sigma_{\Delta}'$ is a simplicial subdivision
of the 
fan $\Sigma_{\Delta}$ such that $\Sigma_{\Delta}(1) \subseteq \Sigma_{\Delta}'(1)$ with 
$\left\vert \Sigma_{\Delta}'(1)\right\vert = \left\vert \Delta \cap N\right\vert -1$, 
where the fan $\Sigma_{\Delta}'$
is obtained by a triangulation of all $2$-dimensional faces $\theta \preceq \Delta$ of $\Delta$
and $\square_{\sigma}^{\circ} = \sigma^{\circ} \cap \square_{\sigma}$
for a cone  $\sigma \in \Sigma_{\Delta}'$.

\smallskip
In each $1$-dimensional cone $\rho \in \Sigma_{\Delta}'(1)$, there exists exactly one lattice point 
$n\in \square_{\rho}^{\circ} \cap N$ with $\kappa_{\Delta} (n)=-1$, because $n$ is a boundary lattice point of  $\Delta$. Each $2$-dimensional cone $\tau \in \Sigma_{\Delta}'(2)$ contains also exactly one 
lattice point $n \in \square_{\tau}^{\circ} \cap N$ with $\kappa_{\Delta} (n)=-2$, because 
the minimal lattice generators $u_{\rho'}$ and $u_{\rho''}$ of $\tau$ are linearly independent 
over $\R$ with $n= u_{\rho'}+u_{\rho''}$. 
These two facts yield
\begin{align}  \label{eq1}
E_{\rm str}\left(X;u,v\right)
&= (uv-1)^3 + \sum_{\rho \in \Sigma_{\Delta}'(1)} (uv-1)^2 
+ \sum_{\tau \in \Sigma_{\Delta}'(2)} (uv-1)  \\ \notag
&+ \sum_{\sigma \in \Sigma_{\Delta}'(3)} \Big(1
+ \sum_{n \in \square_{\sigma}^{\circ} \cap N\atop \kappa_{\Delta} (n) \in \Q \setminus \Z} 
(uv)^{3+\kappa_{\Delta} (n)}\Big),
\end{align}
because each $3$-dimensional cone $\sigma \in \Sigma_{\Delta}'(3)$ has also exactly one lattice point 
$n\in \square_{\sigma}^{\circ} \cap N$ with 
$\kappa_{\Delta} (n)=-3$.
The intersections of all cones $\sigma \in \Sigma'_{\Delta}$ with the $2$-dimensional sphere $S^2$ defines a
triangulation of $S^2$. Therefore, we have 
$\left\vert \Sigma'_{\Delta}(3)\right\vert = 2\cdot \left\vert \Sigma'_{\Delta}(1)\right\vert -4$ and
$\left\vert \Sigma'_{\Delta}(2)\right\vert = 3\cdot \left\vert \Sigma'_{\Delta}(1)\right\vert -6$. 
Using $\left\vert \Sigma_{\Delta}'(1)\right\vert= \left\vert \Delta \cap N\right\vert -1$, we get
\begin{align} \label{eq2}
E_{\rm str}\left(X;u,v\right) \notag
&=  (uv-1)^3 + \left(\left\vert \Delta \cap N\right\vert -1\right)  (uv-1)^2 
+ \sum_{\sigma \in \Sigma_{\Delta}'(2)} (uv-1) +T+ \sum_{\sigma \in \Sigma_{\Delta}'(3)} 1  \\
&= \left((uv)^3+1\right)+ \left(\left\vert \Delta \cap N\right\vert -4\right) \left((uv)^2+ (uv)\right)+ T, 
\end{align}
where 
\[T \defeq  \sum_{\sigma \in \Sigma_{\Delta}'(3)} T_{\sigma} \;\;  \text{ and } \;\;  T_{\sigma}\defeq \sum_{n \in \square_{\sigma}^{\circ} \cap N\atop \kappa_{\Delta} (n) \in \Q \setminus \Z} 
(uv)^{3+\kappa_{\Delta} (n)} \;\;\text{ for } \sigma \in \Sigma_{\Delta}'(3).\]

%%%%%%%%%%%%%%%%%%%%%%%%%%%%%%%%%%%%%%%%%%%%%%%%%%%%%
\begin{figure}[h!]
\centering 
% 1 
\subfigure[]{
\label{fig:picture1}
\scalebox{0.7}{ 	
\includegraphics{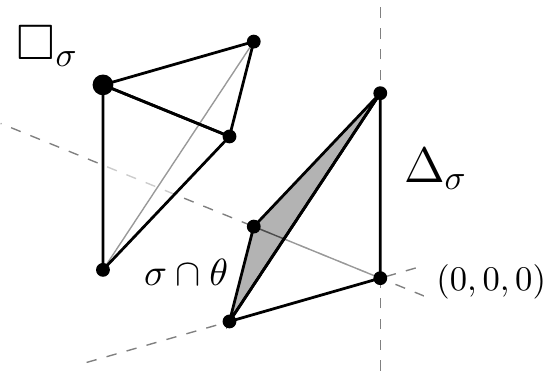}
} 
}
% 2 
\subfigure[]{
\label{fig:picture2}
\scalebox{0.7}{ 
\includegraphics{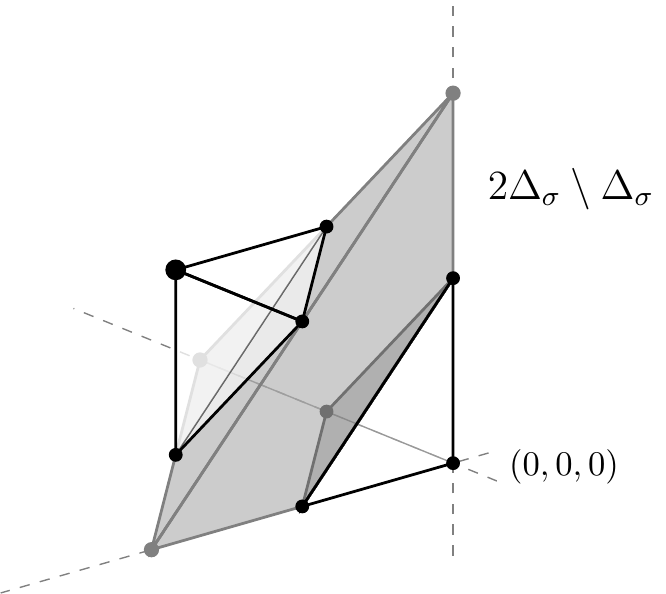}
} 
} 
% 3 
\subfigure[]{ 
\label{fig:picture6}
\scalebox{0.6}{ 
\includegraphics{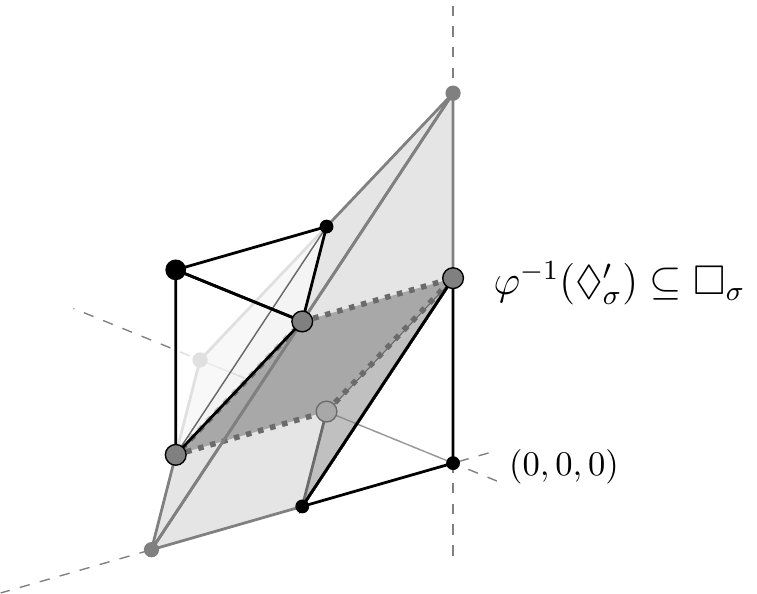}
} 
} 
\caption[]{Illustration to detect  $n \in \square_{\sigma}^{\circ} \cap N$ with  $\kappa_{\Delta} (n) \in \Q \setminus \Z$.}
\label{fig:picture12}
\end{figure}
%%%%%%%%%%%%%%%%%%%%%%%%%%%%%%%%%%%%%%%%%%%%%%%%%%%%%

\smallskip
Let $\sigma \in \Sigma_{\Delta}'(3)$ be a $3$-dimensional cone. Then
$\sigma \cap \theta$ is a $2$-dimensional 
simplex with $v(\sigma \cap \theta) =1$ 
contained in a $2$-dimensional face $\theta \preceq \Delta$ of $\Delta$.
The only lattice points of the $3$-dimensional lattice simplex $\Delta_{\sigma}\defeq \sigma \cap \Delta$
are the origin $0 \in N$ and the three vertices $\nu_1$, $\nu_2$, and $\nu_3$ of the triangle $\sigma \cap \theta$ (Figure \ref{fig:picture1}), because $\Delta$ is 
a canonical Fano polytope.
Therefore, we get
$v(\Delta_{\sigma}) = v(\sigma \cap \theta) \cdot n_{\theta}=n_{\theta}$.

If $n_{\theta}=1$, then $v(\Delta_{\sigma}) =1$ and $\sigma$ is generated by a basis of $N$, i.e.,
$T_{\sigma}=0$. For the following considerations, let $n_{\theta}$ be greater than $1$.
A lattice point $n \in \square_{\sigma}^{\circ} \cap N$ can not be contained in $\Delta_{\sigma}$ and $-\Delta_{\sigma}+(\nu_1+\nu_2+\nu_3)$ (Figure \ref{fig:picture1}) as well as not in one of the three simplices
$\Delta_{\sigma}+\nu_1$, $\Delta_{\sigma}+\nu_2$, and $\Delta_{\sigma}+\nu_3$ (Figure \ref{fig:picture2}). Therefore, $n$ belongs to 
$(2 \Delta_{\sigma})^{\circ} \cap N$, i.e.,
\begin{align} \label{cone3dim}
T_{\sigma}=\sum_{n \in \square_{\sigma}^{\circ} \cap N \atop \kappa_{\Delta} (n) \in \Q \setminus \Z} 
(uv)^{3+\kappa_{\Delta} (n)}
= \sum_{n \in (2 \Delta_{\sigma})^{\circ} \cap N} (uv)^{3+\kappa_{\Delta} (n)} .
\end{align}

To compute $T_{\sigma}$, we apply the Theorem of White \cite{Whi64}, which claims that 
$\Delta_{\sigma}$ 
is affine unimodular isomorphic to a lattice simplex 
$\Delta_{\sigma}' \defeq \conv(\nu_1',\nu_2',\nu_3',\nu_4')$ such that $\nu_1' \defeq (0,0,0)$, 
$\nu_2' \defeq (1,0,0)$, $\nu_3'\defeq (0,0,1)$, and  
$\nu_4'\defeq (a,n_{\theta},1)$ for some integer $a$ with $\text{gcd}(a,n_{\theta})=1$
(Figure \ref{fig:picture3}).
In particular, $\Delta_{\sigma}'$ is contained between the
two planes $\{z=0\} $ and $\{z=1\}$.
The affine isomorphism $\varphi: \Delta_{\sigma} \to \Delta_{\sigma}'$ of White maps
the vertex $0 \in \Delta_{\sigma}$ to the vertex $\nu_4' \in \Delta_{\sigma}'$ and transforms the linear function $\kappa_{\Delta}$ on $\sigma$ into the affine 
linear function $-1+\frac{y}{n_{\theta}}$ on a cone with origin $\nu_4'$ and generaters $\nu_i'-\nu_4'$ ($1 \leq i \leq 3$).

\smallskip
The double simplex $2\Delta_{\sigma}'$ is contained between the
two planes $\{z=0\} $ and $\{z=2\}$. Therefore,
an interior lattice point $n \in (2 \Delta_{\sigma}')^{\circ}$
belongs to an interior point of $\lozenge_{\sigma}' \defeq 2 \Delta_{\sigma}' \cap \{z=1 \}$, which is a
parallelogram spanned by the vectors $(0,1,1)$ and $(a,n_{\theta},1)$ (Figure \ref{fig:picture4}).
We remark that the isomorphism of White maps 
$\lozenge_{\sigma}'$ into a subset of $\square_{\sigma}$
(Figure \ref{fig:picture6}).

%%%%%%%%%%%%%%%%%%%%%%%%%%%%%%%%%%%%%%%%%%%%%%%%%%%%%
\begin{figure} [h!]
\centering 
% 1 
\subfigure[]{
\label{fig:picture3}
\scalebox{0.5}{ 
\includegraphics{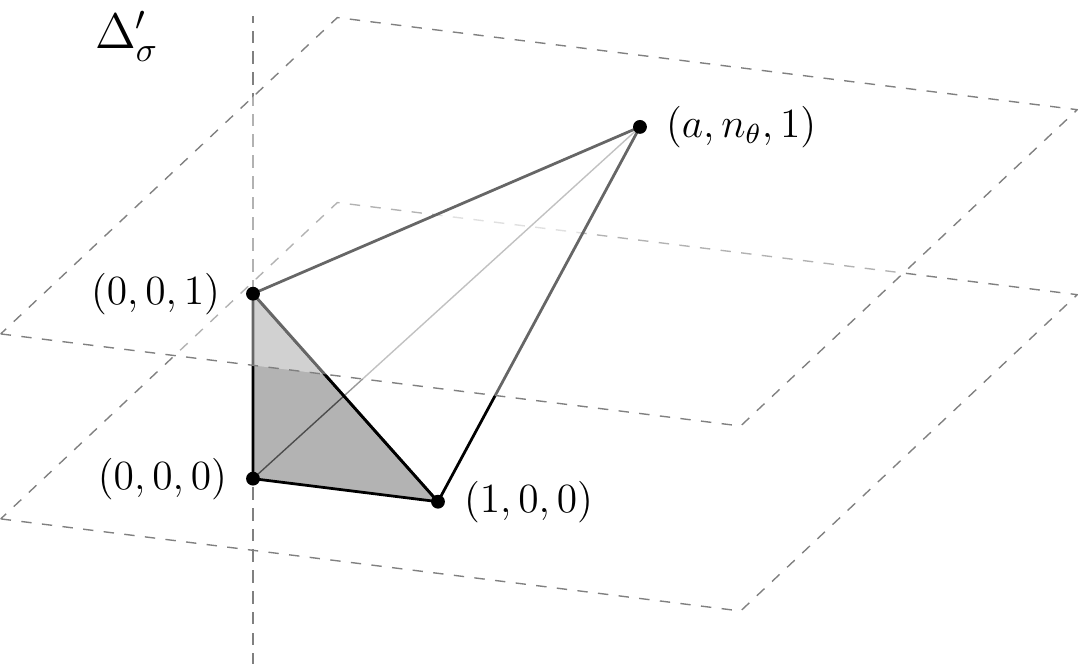}
} 
} 
% 2 
\subfigure[]{
\label{fig:picture4}
\scalebox{0.4}{ 
\includegraphics{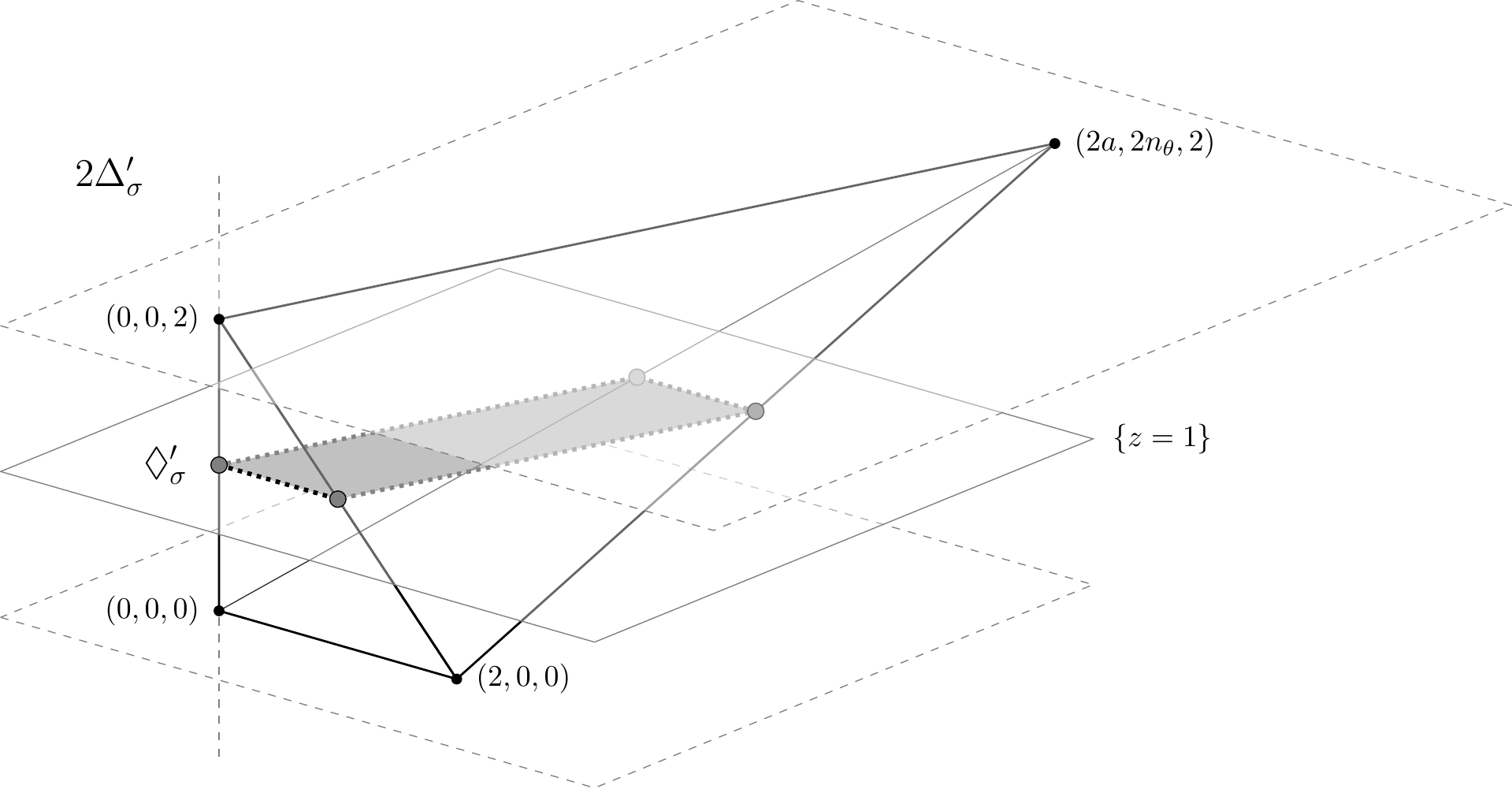}
} 
} 
\caption[]{Theorem of White (here: $a=1$ and $n_{\theta}=3$).}
\label{fig:picture34}
\end{figure}
%%%%%%%%%%%%%%%%%%%%%%%%%%%%%%%%%%%%%%%%%%%%%%%%%%%%%

The volume of the parallelogram $\lozenge_{\sigma}'$ equals $n_{\theta}$ and it has exactly 
$4$ boundary lattice points. By Pick's Theorem, there exist exactly $n_{\theta}-1$ lattice
points in the interior of $\lozenge_{\sigma}'$. 
We claim that the affine linear function $-1+\frac{y}{n_{\theta}}$ takes $n_{\theta}-1$ different 
values on these lattice points (Figure \ref{fig:picture7}). Assume that there exist two lattice points $P=(P_x,P_y)$ and $P'=(P'_x,P'_y)$ with $P'_x \geq P_x$ in the interior of $\lozenge_{\sigma}'$ with the same affine linear function values, 
i.e., $P_y= P'_y$ and $v(\conv(P,P')) \geq 1$. 
Set $E \defeq \conv((0,0),(1,0))$, then
$P'-P \in E^{\circ} \cap N = \emptyset$. Contradiction.

%%%%%%%%%%%%%%%%%%%%%%%%%%%%%%%%%%%%%%%%%%%%%%%%%%%%%
\begin{figure} [h!]
\centering 
\scalebox{0.7}{
\includegraphics{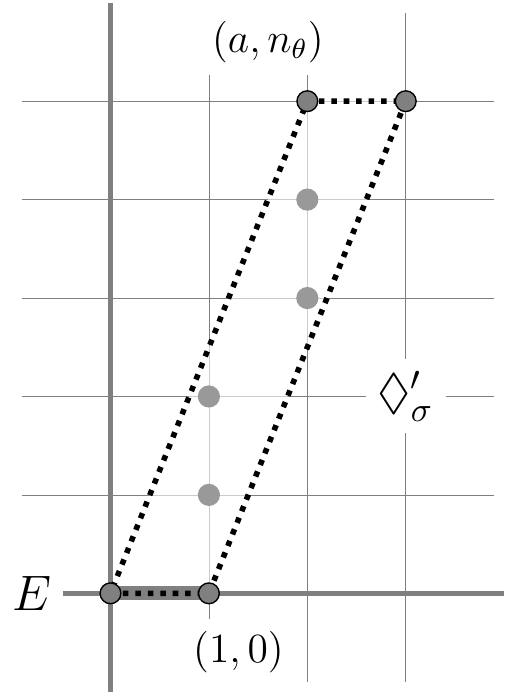}
}
\caption[]{Illustration of $\lozenge_{\sigma}'$ (here: $a=2$ and $n_{\theta}=5$).}
\label{fig:picture7}
\end{figure}
%%%%%%%%%%%%%%%%%%%%%%%%%%%%%%%%%%%%%%%%%%%%%%%%%%%%%

\smallskip
Therefore, we get the set equality
\[\left\{\kappa_{\Delta} (n) \in \Q \setminus \Z \, \middle\vert \, n \in \square_{\sigma}^{\circ} \cap N \right\} 
= \big\{ -1 -\frac{l}{n_{\theta}} \, \big\vert \, l \in \left\{1,\ldots,n_{\theta}-1 \right\} \big\} \]
and Equation \eqref{cone3dim} can be rewritten in the  form
\begin{align*} 
T_{\sigma}=\sum_{n \in \square_{\sigma}^{\circ} \cap N \atop \kappa_{\Delta} (n) \in \Q \setminus \Z} (uv)^{3+\kappa_{\Delta} (n)}
= \sum_{l=1}^{n_{\theta}-1} (uv)^{2-\frac{l}{n_{\theta}}} 
= \sum_{k=1}^{n_{\theta}-1} (uv)^{1+\frac{k}{n_{\theta}}}.
\end{align*}
Using this, the remaining last term in Equation \eqref{eq1} can be computed as
\begin{align} \label{eq3} 
T=\sum_{\sigma \in \Sigma_{\Delta}'(3)}  
\sum_{n \in \square_{\sigma}^{\circ} \cap N \atop \kappa_{\Delta} (n) \in \Q \setminus \Z} (uv)^{3+\kappa_{\Delta} (n) }
&= \sum_{\theta \preceq \Delta \atop \dim(\theta)=2,  n_{\theta}>1}  
v(\theta) \cdot \left( \sum_{k=1}^{n_{\theta}-1}(uv)^{\frac{k}{n_{\theta}}+1} \right), 
\end{align}
because $\sum_{\sigma \cap \theta \subseteq \theta} v(\sigma \cap \theta)= \sum_{\sigma \cap \theta\subseteq \theta}  1=v(\theta)$.

\medskip
A combination of Equation \eqref{eq1}, \eqref{eq2}, and \eqref{eq3} yields the desired result
\begin{align*} 
E_{\rm str}\left(X;u,v\right) = \left((uv)^3+1\right)+ r\cdot \left((uv)^2+ (uv)\right)  + \!\!\!
\sum_{\theta \preceq \Delta \atop \dim(\theta)=2,  n_{\theta}>1}  \!\!\!\!\!
v(\theta) \cdot \left( \sum_{k=1}^{n_{\theta}-1}(uv)^{\frac{k}{n_{\theta}}+1}  \right)  
\end{align*}
with $r= \left\vert \Delta \cap N\right\vert -4$.
\end{proof}

\begin{ex}
Consider $\Delta_1= \conv(e_1,e_2,e_3,-e_1-e_2-e_3)$,
$\Delta_2 = \conv(e_1,e_2,e_3,-e_1-e_2-2e_3)$, 
and $\Delta_3 = \conv(e_1,e_2,e_3,-5e_1-6e_2-8e_3)$
(Figure \ref{fig:picturestrefct}) with $\Delta_i^{\circ} \cap N = \{0 \}$ $(1 \leq i \leq 3)$. 
Therefore, $X_{\Delta_i}$ $(1 \leq i \leq 3)$ are toric Fano threefolds with at worst 
canonical singularities with $X_{\Delta_1} = \P(1,1,1,1) = \P_3$, 
$X_{\Delta_2} = \P(1,1,1,2)$, and $X_{\Delta_3} = \P(1,5,6,8)$.
By using Theorem \ref{strefct} the associated stringy $E$-functions are given as
\begin{align*}
E_{\rm str}\left(X_{\Delta_1};u,v\right) = (uv)^3+(uv)^2+(uv)+1,
\end{align*}
\begin{align*}
E_{\rm str}\left(X_{\Delta_2};u,v\right) 
= (uv)^3 +  (uv)^2+  (uv)^{\frac{3}{2}}+ (uv) +1 ,
\end{align*}
and
\begin{align*}
E_{\rm str}\left(X_{\Delta_3};u,v\right) 
&= (uv)^3+  5\cdot (uv)^2+   
2 \cdot  (uv)^{\frac{5}{3}} +4 \cdot (uv)^{\frac{3}{2}} \\
&\; \; \; \;  +2 \cdot (uv)^{\frac{4}{3}}  +5\cdot  (uv) +1.
\end{align*}
\end{ex}

%%%%%%%%%%%%%%%%%%%%%%%%%%%%%%%%%%%%%%%%%%%%%%%%%%%%%

\section{The number 24 and 3-dimensional canonical Fano polytopes} \label{section3}

The lattice simplex $\Delta_3 = \conv(e_1,e_2,e_3,-5e_1-6e_2-8e_3)$
(Figure \ref{fig:P(1,5,6,8)21} and \ref{fig:P(1,5,6,8)}) is one of the simplest examples of a 
$3$-dimensional canonical Fano polytope, which is not almost reflexive. The canonical Fano
polytope $\Delta_3$ corresponds to the canonical toric Fano threefold $X_{\Delta_3}$, 
which is the weighted projective space $\P(1,5,6,8)$. It was observed
by Corti and Golyshev \cite{CG11} that an affine hypersurface $Z_{\Delta_3} \subseteq (\C^*)^3$ defined
by a general  Laurent polynomial $f_{\Delta_3} \in \C[x_1^{\pm 1}, x_2^{\pm 1}, x_3^{\pm 1} ]$  
with the Newton polytope $\Delta_3$ is birational to an elliptic surface of Kodaira dimension one. 
In particular, 
$Z_{\Delta_3}$  is not birational to a $K3$-surface. The dual polytope $\Delta_3^*$ 
(Figure \ref{fig:P(1,5,6,8)dual22}) is the rational simplex
\[\conv((3,-1,-1),(-1,7/3,-1),(-1,-1,3/2), (-1,-1,-1)).\]
The two rational vertices of $\Delta_3^*$ are dual to 
two $2$-dimensional faces of $\Delta_3$ (highlighted in gray and lightgray in Figure \ref{fig:P(1,5,6,8)21}) 
having lattice distance $2$ and $3$  to the origin,
so that one gets
\[\sum_{\theta \preceq \Delta_3 \atop \dim(\theta) = 2}
\frac{1}{n_{\theta}}\cdot v(\theta)
=  \frac12 \cdot 4+\frac13 \cdot 2+ 1 \cdot 1+1 \cdot 5= \frac{26}{3} \]
and 
\[\sum_{\theta \preceq \Delta_3 \atop \dim(\theta) = 1}
v(\theta) \cdot v(\theta^* )
= 2 \cdot \frac56+ 1 \cdot  \frac52 + 1 \cdot  \frac12 +1 \cdot  \frac{10}{3} + 1 \cdot  \frac23 +1 \cdot  4 = \frac{38}{3}.\]
Therefore, we obtain the equality
 \[v(\Delta_3)
- \sum_{\theta \preceq \Delta_3 \atop \dim(\theta) = 2}
\frac{1}{n_{\theta}}\cdot v(\theta)
+ \sum_{\theta \preceq \Delta_3 \atop \dim(\theta) = 1}
v(\theta) \cdot v(\theta^* )
= 24,\]
because $v(\Delta_3)=20$. 

%%%%%%%%%%%%%%%%%%%%%%%%%%%%%%%%%%%%%%%%%%%%%%%%%%%%%
\begin{figure} [h!]
\centering
% 1
\subfigure[]{
\label{fig:P(1,5,6,8)21}
\scalebox{0.55}{
\includegraphics{pictureP1568}
}
} \hspace{0pt}
% 2
\subfigure[]{
\label{fig:P(1,5,6,8)dual22}
\scalebox{0.6}{
\includegraphics{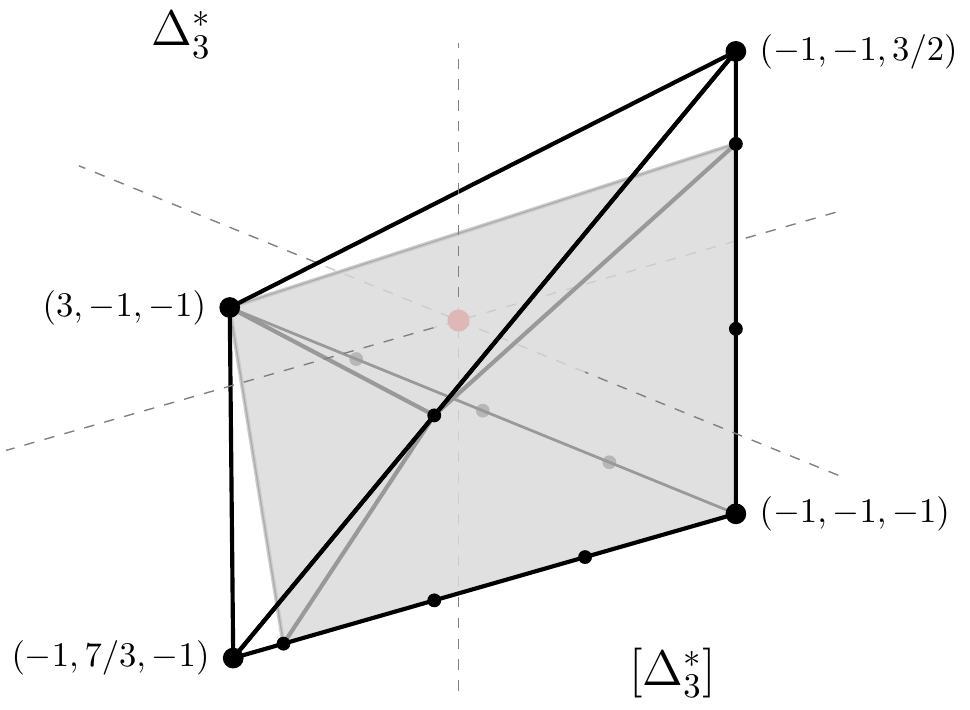}
}
}
\caption[]{Illustration of a not almost reflexive polytope.}
\label{fig:pictureP(1,5,6,8)2}
\end{figure}
%%%%%%%%%%%%%%%%%%%%%%%%%%%%%%%%%%%%%%%%%%%%%%%%%%%%%

\medskip
As we have already mentioned, there exist 
exactly $9,\!089$ canonical Fano polytopes in dimension $3$, which are not almost reflexive. Our purpose 
is now to show that these polytopes also satisfy the identity of Theorem \ref{theo24}, so that one 
gets the further generalization stated in

\setcounter{introtheo}{1}
\begin{introtheo}\label{theo2.2}
Let $\Delta \subseteq N_{\R}$ be an arbitrary $3$-dimensional canonical Fano polytope. Then
\begin{align*}
24= v(\Delta)  - \- \sum_{\theta \preceq \Delta \atop \dim(\theta) = 2}
\frac{1}{n_{\theta}}\cdot v(\theta)
+ \sum_{\theta \preceq \Delta \atop \dim(\theta) = 1}
v(\theta) \cdot v(\theta^*).
\end{align*}
\end{introtheo}

To proof this statement, we can not use the ideas of Section \ref{section1}, 
because there exist $9,\!089$ canonical Fano polytopes $\Delta$ of dimension $3$, that are not related 
to $K3$-surfaces. In contrast, we use a completely different idea based on the stringy Libgober-Wood 
identity and its combinatorial
interpretation for toric varieties containing generalized stringy
Hodge numbers and intersection products of stringy Chern classes.

\smallskip
The original Libgober-Wood identity for smooth projective varieties has appeared in \cite{LW90}
and it has been generalized to the stringy Libgober-Wood identity
\begin{align*}
\frac{d^2}{d u^2} E_{\rm str}\left(X;u,1\right)\Big\vert_{u=1} =
\frac{3d^2-5d}{12} c_d^{\rm str}(X) + \frac{1}{6} c_1(X).c_{d-1}^{\rm str}(X)
\end{align*}
for a $d$-dimensional projective variety $X$ with at worst log-terminal
singularities \cite[Theorem 3.8]{Bat00}, where $c_k^{\rm str}(X) \in A^{k}(X)_{\Q}$ 
denotes the {\em $k$-th stringy Chern class} of $X$
\cite[Section 1]{BS17}.

\smallskip
If $X=X_\Sigma$ is a $d$-dimensional $\Q$-Gorenstein projective toric 
variety of Gorenstein index $q_X$ associated with a $d$-dimensional fan 
$\Sigma$ of rational polyhedral cones in $N_{\R}$ and $-K_X$ an ample $\Q$-Cartier divisor, 
then the stringy $E$-function of $X$ can be written as a finite sum
\begin{align*}
E_{\rm str}\left(X_\Sigma;u,v\right)= \sum_{\alpha \in  [0,d] \cap \frac{1}{q_X} \Z}
\psi_{\alpha}(\Sigma) (uv)^{\alpha}
\end{align*}
\cite[Proposition 4.1]{BS17}, where the coefficients $\psi_{\alpha}(\Sigma)$ are 
nonnegative integers ({\em generalized stringy Hodge numbers}) satisfying the conditions 
$\psi_{0}(\Sigma) = \psi_{d}(\Sigma) =1$ and 
$\psi_{\alpha}(\Sigma) = \psi_{d-\alpha}(\Sigma)$ for all 
$\alpha \in  [0,d] \cap \frac{1}{q_X} \Z$ and the 
{\em Gorenstein index} $q_X$ is the smallest positive
integer such that $q_XK_X$ is a Cartier divisor.
Moreover, $\Delta_{-K_X}\subseteq M_{\R}$ denotes the 
$d$-dimensional rational
polytope corresponding to the anticanonical divisor $-K_X$ of $X$ defined as
\begin{align*}
\Delta_{-K_X} \defeq 
\{ y \in M_{ \R} \, \vert \, \left<y,u_{\rho} \right> \geq -1 \; 
\forall \rho \in \Sigma_{\Delta}(1) \} \subseteq M_{\R}
\end{align*}
and $\Delta_{-K_X}^{\sigma}\preceq \Delta_{-K_X}$ a $k$-dimensional face of 
$\Delta_{-K_X}$ corresponding to a cone $\sigma \in \Sigma(d-k)$ defined as
\begin{align*}
\Delta_{-K_X}^{\sigma} \defeq 
\{ y \in \Delta_{-K_X}  \, \vert \, \left<y,u_{\rho} \right> = -1 
\; \forall \rho \in \Sigma_{\Delta}(1) \text{ with } \rho \in \sigma \}
\end{align*}
\cite[Theorem 3.5]{BS17},
where $u_{\rho} \in N$ denotes
the primitive ray generator of a $1$-dimensional cone 
$\rho \in \Sigma_{\Delta}(1)$.

\smallskip
In this setting, the stringy Libgober-Wood identity is equivalent to the 
combinatorial identity stated in

\begin{prop}[{\cite[Theorem 4.3]{BS17}}] \label{prop3}
Let $X\defeq X_\Sigma$ be a $d$-dimensional $\Q$-Gorenstein projective toric 
variety of Gorenstein index $q_X$ associated with a fan 
$\Sigma$ of rational polyhedral cones in $N_{\R}$ and $-K_X$ an ample $\Q$-Cartier divisor. Then 
\begin{align*} 
\sum_{\alpha \in [0, d] \cap
\frac{1}{q_X}\Z} \psi_{\alpha} \left(\Sigma\right)
\left(\alpha - \frac{d}{2} \right)^2 = \frac{d}{12} v(\Sigma)
+ \frac{1}{6} \sum_{\sigma \in \Sigma(d-1)}v(\sigma)
\cdot v\left(\Delta_{-K_X}^{\sigma} \right),
\end{align*}
where $v(\Sigma) \defeq \sum_{\sigma \in \Sigma(d)} v(\sigma)$ denotes the 
{\em normalized volume of a fan}
$\Sigma$ and $v(\sigma)$ the {\em normalized volume of a cone} $\sigma$ defined
as the normalized volume of the polytope obtained as the convex hull of the origin
and all primitive ray generators of the cone $\sigma \in \Sigma$.
\end{prop}

\begin{proof}[Proof of Theorem \ref{theo2.2}]
We apply Proposition \ref{prop3}
to the canonical toric Fano threefold $X \defeq X_{\Delta}$
with at worst canonical singularities
corresponding to the canonical
Fano polytope $\Delta$ of dimension $3$. In this case, the associated fan $\Sigma_{\Delta}$ 
is the spanning fan of $\Delta$, i.e., it consists of cones
$\sigma_\theta = \R_{\geq 0} \theta$, where $\theta \preceq \Delta$ runs over all faces $\theta$ of $\Delta$.

By construction of the fan $\Sigma_{\Delta}$, it is easy to show that $v(\Sigma_{\Delta})= v(\Delta)$.
Moreover,
 \begin{align*}
v(\sigma_\theta)
= v(\theta)
\end{align*}
for any $1$-dimensional face
 $\theta \preceq \Delta$, because $v(\sigma_\theta) = v({\rm conv}(0, \theta)) = v(h_\theta) \cdot
 v(\theta)$, 
where $h_{\theta}$ denotes the height of the lattice 
triangle  $\conv(0, \theta) $ with base $\theta$ and
$v(h_{\theta})=1$, because all nonzero lattice points 
of the lattice triangle $\conv(0, \theta) $ are contained in its side $\theta$. 

By definition of the dual polytope $\Delta^*$, every $1$-dimensional face $\theta^*$ of $\Delta^*$ has lattice distance $1$ to the origin. So one has
$\Delta_{-K_X} = \Delta^*$ and $\Delta_{-K_X} ^{\sigma_{\theta}} = \theta^*$
yielding
\[v(\Delta_{-K_X} ^{\sigma_{\theta}} )=  v(\theta^*). \]

Therefore, the right hand side of the  stringy Libgober-Wood identity in Proposition \ref{prop3}
equals
\[\frac{3}{12} v(\Delta)
+  \frac{1}{6}  \sum_{\theta \preceq \Delta \atop \dim(\theta)=1} v(\theta)
\cdot v\left(\theta^* \right) \]
and the number $v(\Delta)$ equals
\[E_{\rm str}(X; 1,1) = 2 + 2 \cdot  \left( \left\vert \Delta \cap N\right\vert -4\right) +
\sum_{\theta \preceq \Delta \atop \dim(\theta)=2,  n_{\theta}>1}
v(\theta) \cdot  \sum_{k=1}^{n_{\theta}-1} 1 \]
\cite[Proposition 4.10]{Bat98}. 
The left hand side of the stringy Libgober-Wood identity in Proposition \ref{prop3} can be computed
by Theorem \ref{strefct} as
\begin{align*}
 2 \cdot 1 \left(\frac32\right)^2 + 2 \cdot  \left( \left\vert \Delta \cap N\right\vert -4\right) \left(1- \frac32\right)^2
+\sum_{\theta \preceq \Delta \atop \dim(\theta)=2,  n_{\theta}>1}
v(\theta) \cdot  \sum_{k=1}^{n_{\theta}-1} 1 \cdot
\left( \frac{k}{n_{\theta}} - \frac{1}{2} \right)^2.
\end{align*}
Comparing right and left hand sides, a short calculation yields
\begin{align*}
&24
=  6 \cdot \! \! \! \! \! \! \! \! \sum_{\theta \preceq \Delta \atop \dim(\theta)=2,  n_{\theta}>1}
\! \! \! \! \! \!  v(\theta) \cdot
\sum_{k=1}^{n_{\theta}-1}
\left(   \frac{1}{4}-\left(\frac{k}{n_{\theta}} - \frac{1}{2} \right)^2 \right) +
 \sum_{\theta \preceq \Delta \atop \dim(\theta)=1} v(\theta) \cdot v\left({\theta}^* \right)
 \end{align*}
 and using Lemma \ref{lemgauss} from below adds up to
 \begin{align*}
  24 &= \sum_{\theta \preceq \Delta \atop \dim(\theta)=2,  n_{\theta}>1}
 \! \! \! \! \! \!  v(\theta)   \left( n_{\theta} - \frac{1}{n_{\theta}}\right)
  + \sum_{\theta \preceq \Delta \atop \dim(\theta)=1} v(\theta) \cdot v\left({\theta}^* \right) \\
    \Leftrightarrow
  24 &= v(\Delta)  - \sum_{\theta \preceq \Delta \atop \dim(\theta)=2}\frac{1}{n_\theta} \cdot v(\theta) 
+ \sum_{\theta \preceq \Delta \atop \dim(\theta)=1} v(\theta) \cdot v\left({\theta}^* \right) ,
\end{align*}
because $v(\Delta) = \sum_{\theta \preceq \Delta \atop \dim(\theta)=2} v(\theta) \cdot  n_{\theta} $. 
\end{proof}

\begin{lem} \label{lemgauss}
Let $n$ be a positive integer. Then 
\[6  \cdot \sum_{k=1}^{n-1} 
\left( \frac{1}{4}- \left(\frac{k}{n}  - \frac{1}{2} \right)^2  \right) 
=n -\frac{1}{n}. \]
\end{lem}

\begin{proof}
Applying the well-known formulas for $\sum_{k=1}^{n-1} k$ and $\sum_{k=1}^{n-1} k^2$, 
we get
\begin{align*}
 6  \cdot \sum_{k=1}^{n-1} &
\left( \frac{1}{4}- \left(\frac{k}{n}  - \frac{1}{2} \right)^2  \right) 
= 6  \cdot \sum_{k=1}^{n-1} 
\left( \frac{1}{4}- \frac{k^2}{n^2}  + \frac{k}{n}  -\frac{1}{4}  \right)  
=  -\frac{6}{n^2} \cdot \sum_{k=1}^{n-1} 
k^2 +  \frac{6}{n} \sum_{k=1}^{n-1} k \\
&=  -\frac{1}{n^2} [(n-1)n(2(n-1)+1)] +  \frac{3}{n} [(n-1)n] 
= n -\frac{1}{n}.
\end{align*}
\end{proof}

%%%%%%%%%%%%%%%%%%%%%%%%%%%%%%%%%%%%%%%%%%%%%%%%%%%%%%%%%%%%%%%%%%%%%%%%%%%%%%%%%%%%%%%%%%%%%%%%%%%%
%\bibliographystyle{amsalpha}
%\bibliography{mylit}
\newcommand{\etalchar}[1]{$^{#1}$}

\end{document}